\documentclass[10pt]{amsart}
\usepackage{amsmath, amssymb,amsthm}
\usepackage{typearea}

\typearea{16}
\title[The cohomology of framed moduli spaces ]{The cohomology of framed moduli spaces and the coordinate ring of torus fixed points of quotient singularities }
\author{Kohei Hatano}
\date{}

\newtheorem{prop}{Proposition}[subsection]
\newtheorem{lem}[prop]{Lemma}
\newtheorem{theorem}[prop]{Theorem}
\newtheorem{cor}[prop]{Corollary}
\newtheorem{dfn}[prop]{Definition}
\newtheorem{exa}[prop]{Example}
\newtheorem{conj}[prop]{Conjecture}
\newtheorem{rem}[prop]{Remark}
\begin{document}
\allowdisplaybreaks[4]
\begin{abstract}
If two conical symplectic resolutions $X\to X_0$ and $X^!\to X_0^!$ are symplectic dual, the cohomology ring $H^*(X)$ and the coordinate ring of the $\mathbb{T}$-fixed point scheme in $X_0^!$ are expected to be isomorphic as graded algebras. This statement is called the Hikita conjecture, and the conjecture is known in several cases. In this paper, we study the cohomology of framed moduli spaces over the projective plane and the coordinate ring of the $\mathbb{T}$-fixed point scheme of $\mathbb{C}^{2n}/((\mathbb{Z}/r\mathbb{Z})\wr S_n)$ and show that they are isomorphic as graded vector spaces. 
\end{abstract}

\maketitle

\section{Introduction}
Let $G$ be the general linear group of rank $n$ over $\mathbb{C}$. The symbols $B$ and $T$ denote a Borel subgroup and a maximal torus of $G$, respectively. We denote the Lie algebras of $G$, $B$ and $T$ by $\mathfrak{g}$, $\mathfrak{b}$ and $\mathfrak{t}$, respectively. Let $\mathcal{N}$ be the nilpotent cone of $\mathfrak{g}$. It is well known that the Springer resolution is given by 
\[\mu: T^*(G/B)\simeq \{ (x,gB)\in \mathcal{N}\times G/B\mid \mathrm{Ad}^{-1}(g)x\in \mathfrak{b}\}\ni (x,gB)\mapsto x\in \mathcal{N}.\]
For $e\in \mathcal{N}$, the fiber $\mu^{-1}(e)$ is called the Springer fiber. Let $\mathcal{O}_{\lambda}$ denote the nilpotent orbit corresponding to a partition $\lambda$ of $n$. Then the algebraic realization of the cohomology ring of the Springer fiber is given by the following theorem. 
\begin{theorem}[\cite{deconcini,tanisaki}]
Let $\lambda$ be a partition of $n$. For $e\in \mathcal{O}_{\lambda}$, we have the following isomorphism as graded algebras:
\[H^*(\mu^{-1}(e),\mathbb{C}) \simeq \mathbb{C}[\mathfrak{t}\cap \overline{\mathcal{O}_{\lambda^{T}}}],\]
where $\lambda^{T}$ is the transpose of $\lambda$ and $\mathfrak{t}\cap \overline{\mathcal{O}_{\lambda^{T}}}$ is the scheme-theoretic intersection of $\mathfrak{t}$ and $\overline{\mathcal{O}_{\lambda^{T}}}$. 
\end{theorem}
This theorem is remarkable in that the cohomology ring, which is defined in
topological terms, can recover information about an algebraic variety.
In \cite{hikita}, Hikita conjectured that this theorem should be generalized to pairs of conical symplectic resolutions which are symplectic dual to each other. Symplectic duality was introduced in \cite{conical}. In this paper, the symbols $\mathbb{S}$ and $\mathbb{T}$ denote copies of the multiplicative group $\mathbb{C}^*$.
\begin{dfn} A smooth symplectic variety $(X,\omega)$ over $\mathbb{C}$ with an $\mathbb{S}$-action is said to be conical if and only if it satisfies the following conditions:
\begin{enumerate}
\item[\textup{(1)}] The coordinate ring $\mathbb{C}[X]$ has a weight decomposition with non-negative weights and $\mathbb{C}[X]_0=\mathbb{C}$.  
\item[\textup{(2)}] The projection $X\to X_0=\mathrm{Spec}\:\mathbb{C}[X]$ is projective and birational. 
\item[\textup{(3)}] There exists $d\in \mathbb{N}$ with $s^*\omega=s^d\omega$ for any $s\in\mathbb{S}$. 
\end{enumerate}
We call $X\to X_0$ a conical symplectic resolution.
\end{dfn}

The Springer resolution is a conical symplectic resolution. 
Some pairs of conical symplectic resolutions are considered to be symplectic dual. We give a list of examples of symplectic dual pairs. 
\begin{exa}\leavevmode
\begin{enumerate}
\item[\textup{(1)}] \textup{The resolution $\pi:\mathrm{Hilb}^n(\mathbb{C}^2)\to \mathbb{C}^{2n}/S_n$ is self-dual; here $\pi$ is the Hilbert-Chow morphism.}
\item[\textup{(2)}] \textup{Let $M_{r,n}$ be the framed moduli space of pairs $(E,\Phi)$, where $E$ is a torsion-free sheaf on $\mathbb{P}^2$ with $\mathrm{rank}\: E=r$ and $c_2(E)=n$, and $\Phi:E|_{l_{\infty}}\to \mathcal{O}_{\mathbb{P}^2}^{\oplus r}|_{l_{\infty}}$ is an isomorphism, where $l_{\infty}=\{[0,z_1,z_2]\in \mathbb{P}^2\}$.
By \cite[Corollary 10.11]{conical}, the conical symplectic resolution $\mathrm{Hilb}^n(\widetilde{\mathbb{C}^{2}/(\mathbb{Z}/r\mathbb{Z})})\to \mathbb{C}^{2n}/((\mathbb{Z}/r\mathbb{Z})\wr S_n)$ is symplectic dual to the resolution $M_{r,n}\to (M_{r,n})_0$, where $\widetilde{\mathbb{C}^{2}/(\mathbb{Z}/r\mathbb{Z})}$ denotes the minimal resolution of $\mathbb{C}^{2}/(\mathbb{Z}/r\mathbb{Z})$.}
\item[\textup{(3)}] \textup{Let $G$ be a finite subgroup of $\mathrm{SL}_2(\mathbb{C})$. Then the crepant resolution $\pi:\widetilde{\mathbb{C}^2/G}\to\mathbb{C}^2/G $ is symplectic dual to the resolution of the closure of the minimal nilpotent orbit in $\mathfrak{g}$, where  $\mathfrak{g}$ is the simply laced simple  Lie algebra corresponding to the subgroup $G$.}
\item[\textup{(4)}] \textup{Let $X$ and $X'$ be polarized hyperplane arrangements. If $X$ and $X'$ are Gale dual, the associated hypertoric varieties $\mathfrak{M}(X)$ and  $\mathfrak{M}(X')$ are symplectic dual.}
\item[\textup{(5)}] \textup{An $A$-type Spaltenstein variety is symplectic dual to another  $A$-type Spaltenstein variety. More generally, an affine Grassmannian slice and a quiver variety  of type $ADE$ are symplectic dual (\cite{conical}).}
\end{enumerate}
\end{exa}

\begin{conj}[\cite{hikita}]
Assume that $X$ and $X^!$ are equipped with a $\mathbb{T}$-action commuting with the $\mathbb{S}$-action. If $X\to X_0$ and $X^!\to X_0^!$ are symplectic dual, then the cohomology ring $H^*(X^!)$ and the coordinate ring of the $\mathbb{T}$-fixed point scheme in $X_0$ are expected to be isomorphic as graded algebras, and vice versa. Here the grading of the coordinate ring is induced by the $\mathbb{S}$-action on $X$. 
\end{conj}
In \cite[Theorem 1.1, Theorem A.1, Theorem B.1]{hikita}, this conjecture has been proved in cases (1), (4), and the first part of case (5).  In \cite{pavel}, Shlykov proved that the cohomology ring of the crepant resolution of $\mathbb{C}^2/G$, where $G$ is a finite subgroup of $\mathrm{SL}_2(\mathbb{C})$, is isomorphic to the coordinate ring of the fixed point scheme of the minimal nilpotent orbit closure; this corresponds to case (3).  
In \cite[Theorem 1.5]{affine}, it was proved that the coordinate ring of the fixed point scheme of the affine Grassmannian slice is isomorphic to the cohomology ring of  the type $ADE$ quiver variety. 
In this paper, we consider the Hikita conjecture in case (2).
In this case, while the structures of the cohomology rings  $H^*(M_{r,n},\mathbb{C})$ 
and  $H^*(\mathrm{Hilb}^n(\widetilde{\mathbb{C}^{2}/(\mathbb{Z}/r\mathbb{Z})}),\mathbb{C})$ are well known, the structure of the coordinate rings of the fixed point schemes is not yet sufficiently explicit. Hence, in case (2), we compare these objects as graded vector spaces.

Our purpose is to prove the following theorem.
\begin{theorem}\label{hatano}
We have a non-canonical isomorphism as graded vector spaces:
\[
H^*(M_{r,n},\mathbb{C})\simeq
\mathbb{C}\!\left[
\left(
\mathbb{C}^{2n}/\bigl((\mathbb{Z}/r\mathbb{Z})\wr S_n\bigr)
\right)^{\mathbb{T}}
\right].
\]
Here $\mathbb{T}\simeq\mathbb{C}^*$, and the grading on the coordinate ring is induced by the $\mathbb{S}$-action.

\end{theorem}
\begin{rem}
\textup{If $r=1$, the framed moduli space $M_{1,n}$ coincides with the Hilbert scheme $\mathrm{Hilb}^n(\mathbb{C}^2)$. Hence, in this case, the isomorphism preserves the algebra structure by \cite{hikita}. \\
If $n=1$, Theorem \ref{hatano} gives the opposite isomorphism to that of \cite{pavel} in type $A$; namely, we have the following isomorphism:   }
\[\mathbb{C}[(\mathbb{C}^2/(\mathbb{Z}/r\mathbb{Z}))^{\mathbb{T}}]=\mathbb{C}[z]/z^r=H^*(M_{r,1}). \]

\end{rem}

\subsection*{Organization}
The paper is organized as follows. First, we consider the projective limit $S$ of the coordinate rings $\mathbb{C}\!\left[\left(\mathbb{C}^{2n}/\bigl((\mathbb{Z}/r\mathbb{Z})\wr S_n\bigr)\right)^{\mathbb{T}}\right]$ and the ideal $J$ defining the $\mathbb{T}$-fixed point schemes in the limit. By Lemma \ref{basis}, we obtain a basis of $S/J$ indexed by $r$-tuples of partitions. We then study the relations among the generators in $S/J$ in Lemma \ref{lem1} and Subsection~2.2. By Theorem \ref{vec}, we obtain a basis of the coordinate ring $\mathbb{C}[Z_n^{\mathbb{T}}]$ indexed by $r$-tuples of partitions $(\lambda_0,\cdots,\lambda_{r-1})$ satisfying $|\lambda|+l(\lambda)\leq n$.
\subsection*{Acknowledgments}
I would like to thank Daisuke Matsushita for helpful discussions on the manuscript of this paper. I am also grateful to Noriyuki Abe, Syu Kato and Tatsuyuki Hikita for their useful comments and for teaching me about the Hikita conjecture.

\section{Coordinate ring of quotient singularities }
\subsection{Structure of projective limits}

We denote by $Z_n$ the quotient of $\mathbb{C}^{2n}$ by the group $G_n$, where $G_n=(\mathbb{Z}/r\mathbb{Z})\wr S_n$. We define the $\mathbb{S}$-action on $\mathbb{C}^{2n}$ by $s((X_1,Y_1),\cdots,(X_n,Y_n))=((sX_1,sY_1),\cdots,(sX_n,sY_n))$ for any $s\in \mathbb{S}$ and $((X_1,Y_1),\cdots,(X_n,Y_n))\in \mathbb{C}^{2n}$. Since this action commutes with the action of $G_n$ on $\mathbb{C}^{2n}$, we obtain the $\mathbb{S}$-action on $Z_n$. We consider the coordinate ring $\mathbb{C}[Z_n]$. By the definition of $Z_n$, we have
\begin{align}
\mathbb{C}[Z_n]&=\mathbb{C}[X_1,\cdots,X_n,Y_1,\cdots,Y_n]^{G_n}\notag \\
&\simeq (\mathbb{C}[x_1,\cdots,x_n,y_1,\cdots,y_n,z_1,\cdots,z_n]/(x_1y_1-z_1^r,\cdots,x_ny_n-z_n ^r))^{S_n}.  \notag
\end{align}
Set $R_n=\mathbb{C}[x_1,\cdots,x_n,y_1,\cdots,y_n,z_1,\cdots,z_n]$ and $I_n=(x_1y_1-z_1^r,\cdots,x_ny_n-z_n ^r)$. 
The $\mathbb{S}$-action on $Z_n$ induces the $\mathbb{S}$-action on $\mathbb{C}[x_i,y_i,z_i]/I_n$ given by $s\cdot x_i=s^rx_i$, $s\cdot y_i=s^ry_i$, $s\cdot z_i=s^2z_i$ for $s\in\mathbb{S}$. 
\begin{prop}\label{syou}
Here $I_n^{S_n}$ denotes $I_n\cap R_n^{S_n}$. Then
\[\mathbb{C}[Z_n]\simeq R_n^{S_n}/I_n^{S_n}. \]
\end{prop}
\begin{proof}
Since we work over $\mathbb{C}$, the Reynolds operator
\[
\rho(f)=\frac{1}{n!}\sum_{\sigma\in S_n}\sigma f
\]
shows that taking $S_n$-invariants is exact. Therefore the natural morphism
\[
R_n^{S_n}/(I_n\cap R_n^{S_n})\longrightarrow (R_n/I_n)^{S_n}
\]
is an isomorphism. 
\end{proof}
We consider the natural projection
\[\phi_n:R_{n+1}^{S_{n+1}}\to R_n^{S_n}, \]
which is induced by 
\[f(x_1,\cdots,x_n,x_{n+1},y_1,\cdots,y_n,y_{n+1},z_1,\cdots,z_n,z_{n+1})\to f(x_1,\cdots,x_n,0,y_1,\cdots,y_n,0,z_1,\cdots,z_n,0).\]
Set $S'=\projlim_n R_n^{S_n}$ and
\[
S=\projlim_n R_n^{S_n}/(I_n\cap R_n^{S_n})=\projlim_n \mathbb{C}[Z_n].
\]
Let
\[
I^{\mathrm{inv}}:=\ker(S'\longrightarrow S).
\]
Then $S\simeq S'/I^{\mathrm{inv}}$. Equivalently, in each finite degree this kernel is induced by $I_n\cap R_n^{S_n}$ for $n$ sufficiently large.
For any tri-partition $\Lambda=(a_1,b_1,c_1)(a_2,b_2,c_2)\cdots(a_l,b_l,c_l)$, the symmetric function $m_\Lambda$ is the symmetrization of the monomial $x_1^{a_1}y_1^{b_1}z_1^{c_1}\cdots x_l^{a_l}y_l^{b_l}z_l^{c_l}$ and we set $l(\Lambda)=l$. Throughout, a tri-partition does not contain the factor $(0,0,0)$.
The set $\{m_{\Lambda}\mid \Lambda\}$ is a basis of $S'$ as a vector space. As an algebra, $S'$ is freely generated by $\{m_{(a,b,c)}\mid (a,b,c)\neq(0,0,0)\}$ by \cite{sym}. Hence $\{m_{\Lambda}\mid \Lambda=(a_1,b_1,c_1)(a_2,b_2,c_2)\cdots , 0\leq c_i \leq r-1 \}$ is a basis of $S$ and $S\simeq \mathbb{C}[m_{(a,b,c)}]_{0\leq c\leq r-1,(a,b,c)\neq(0,0,0)}$.
\begin{lem}\label{lem1}
We have the following relations in $S$:
\begin{align}
m_{(a,b,c)}m_{\Lambda}=&(u_{(a,b,c)}+1)m_{(a,b,c)\Lambda}+\sum_{\substack{u_{(i,j,k)}>0\ c+k\leq r-1}}(u_{(a+i,b+j,c+k)}+1)m_{(a+i,b+j,c+k)\Lambda-(i,j,k)} \notag \\
&+\sum_{\substack{u_{(i,j,k)}>0\ c+k\geq r}}(u_{(a+i+1,b+j+1,c+k-r)}+1)m_{(a+i+1,b+j+1,c+k-r)\Lambda-(i,j,k)}, \notag
\end{align}
where the non-negative integer $u_{(i,j,k)}$ is the number of $(i,j,k)$ in $\Lambda$. 
\end{lem}
\begin{proof}
By \cite[Lemma 2.4]{hikita}, we have
\[m_{(a,b,c)}m_{\Lambda}=(u_{(a,b,c)}+1)m_{(a,b,c)\Lambda}+\sum_{u_{(i,j,k)}>0}(u_{(a+i,b+j,c+k)}+1)m_{(a+i,b+j,c+k)\Lambda-(i,j,k)}.  \]
In $S$, we have the following relation
\[x_i^ay_i^bz_i^{c+k}=x_i^{a+1}y_i^{b+1}z_i^{c+k-r}\quad (c+k\geq r).\]
Hence, we have the desired relations. 
\end{proof}
We define the $\mathbb{T}$-action on $\mathbb{C}^{2n}$ by  $t((X_1,Y_1),\cdots,(X_n,Y_n))=((tX_1,t^{-1}Y_1),\cdots,(tX_n,t^{-1}Y_n))$ for any $t\in \mathbb{T}$ and $((X_1,Y_1),\cdots,(X_n,Y_n))\in\mathbb{C}^{2n}$. Since this action commutes with the $G_n$-action on $\mathbb{C}^{2n}$, we obtain the $\mathbb{T}$-action on $Z_n$. Let
\[
J:=\ker\left(S\longrightarrow \projlim_n \mathbb{C}[Z_n^{\mathbb{T}}]\right).
\]
Equivalently, $J$ is the ideal generated by the non-zero $\mathbb{T}$-weight components of $S$. Since $m_{(a,b,c)}$ has non-zero $\mathbb{T}$-weight precisely when $a\neq b$, by \cite{fixed} we have
\[
J=(m_{(a,b,c)}\mid a\neq b).
\]
Therefore,
\[
S/J\simeq \mathbb{C}[\bar m_{(a,a,c)}]_{0\leq c\leq r-1,(a,c)\neq(0,0)}.
\]
For later use, let $K_n\subset S$ be the ideal generated by all $m_\Lambda$ with $l(\Lambda)>n$.
\begin{prop}\label{finitequot}
\[
S/(K_n+J)\simeq \mathbb{C}[Z_n^{\mathbb{T}}].
\]
\end{prop}
\begin{proof}
Let $\pi_n:S\to \mathbb{C}[Z_n]$ be the specialization map induced by setting all variables with index greater than $n$ equal to zero. With respect to the stable monomial basis, $\pi_n(m_\Lambda)=0$ if $l(\Lambda)>n$, while the classes with $l(\Lambda)\leq n$ map to the corresponding finite-level monomial symmetric functions. Hence
\[
\ker \pi_n=K_n.
\]
The map $\pi_n$ is $\mathbb{T}$-equivariant. Therefore the image of the ideal $J$ generated by the non-zero $\mathbb{T}$-weight components of $S$ is precisely the ideal $J_n$ generated by the non-zero $\mathbb{T}$-weight components of $\mathbb{C}[Z_n]$. It follows that
\[
S/(K_n+J)\simeq \mathbb{C}[Z_n]/J_n.
\]
Finally, if an affine scheme $Y=\operatorname{Spec} A$ is equipped with a torus action and $A=\bigoplus_\chi A_\chi$ is the corresponding weight decomposition, then the scheme-theoretic fixed point subscheme $Y^{\mathbb{T}}$ is defined by the ideal generated by all non-zero weight spaces $A_\chi$ with $\chi\neq 0$. Applying this to $Y=Z_n$ gives
\[
\mathbb{C}[Z_n]/J_n\simeq \mathbb{C}[Z_n^{\mathbb{T}}],
\]
and hence the assertion follows.
\end{proof}
We use the following convention for partitions. The partition $\lambda_0$ has only positive parts, while for $1\leq i\leq r-1$ the partition $\lambda_i$ is allowed to have finitely many zero parts. Thus we write
\[
\lambda_0=(1^{\alpha_{0,1}},2^{\alpha_{0,2}},\cdots),\qquad
\lambda_i=(0^{\alpha_{i,0}},1^{\alpha_{i,1}},2^{\alpha_{i,2}},\cdots)\quad (1\leq i\leq r-1),
\]
where all but finitely many $\alpha_{i,j}$ are zero. We set
\[
l(\lambda_i)=\sum_{j\geq 0}\alpha_{i,j},\qquad
|\lambda_i|=\sum_{j\geq 1}j\alpha_{i,j}
\]
with the convention $\alpha_{0,0}=0$. Zero parts for $i>0$ are part of the data; they correspond to the factors $m_{(0,0,i)}$. The excluded zero parts of $\lambda_0$ will be accounted for by the inequality $|\lambda|+l(\lambda)\leq n$ in Theorem \ref{vec}.
For any such $r$-tuple $\lambda=(\lambda_0,\cdots,\lambda_{r-1})$, we define the associated tri-partition by
\begin{align}
\lambda&=(\lambda_0,\cdots,\lambda_{r-1})=(\lambda_0^1,0,0)\cdots(\lambda_0^{l(\lambda_0)},0,0)\cdots(\lambda_{1}^1,0,1)\cdots(\lambda_i^j,0,i)\cdots(\lambda_{r-1}^{l(\lambda_{r-1})},0,r-1) \notag \\
&=\prod_{i=0}^{r-1}\prod_{j=1}^{l(\lambda_i)}(\lambda_i^j,0,i).\notag
\end{align}
We consider $S/J$ as the graded algebra whose grading is induced by the $\mathbb{S}$-action on $S$. For instance, $\deg \bar m_{(a,a,c)}=2(ra+c)$.
We set $l(\lambda)=\sum_il(\lambda_i)$ and $|\lambda|=\sum_i|\lambda_i|$. In the following, we just consider a tri-partition $\Lambda=(a_1,b_1,c_1)(a_2,b_2,c_2)\cdots$ with $0\leq c_i\leq r-1$ for any $i$.
\begin{lem}\label{basis}
The set $\{\bar{m}_{(\lambda_{0},\cdots,\lambda_{r-1})(0,1,0)^{|\lambda|}}\}$ is a basis of $S/J$.
\end{lem}
\begin{proof} We prove the lemma in three steps.
\begin{enumerate}
\item [Step 1)] We claim the set $\{\bar{m}_{(a_1,a_1,c_1)(a_2,a_2,c_2)\cdots}\}$ spans $S/J$. It is enough to show that any $\bar m_{\Lambda}$ is written as a linear combination of $\bar{m}_{(a_1,a_1,c_1)(a_2,a_2,c_2)\cdots}$. 
Suppose there exists $(a,b,c)$ contained in $\Lambda$ with $a\neq b$. 
In the case of $l(\Lambda)=1$, $\bar m_{(a,b,c)}=0$. Suppose $l(\Lambda)>1$.
By Lemma \ref{lem1}, we have
\[0=\bar m_{(a,b,c)}\bar m_{\Lambda-(a,b,c)}=u\bar m_{\Lambda}+\sum_{l(\Gamma)=l(\Lambda)-1}u_{\Gamma}\bar m_{\Gamma}\]
for some $u>0$, $u_{\Gamma}\in \mathbb{Z}_{\geq 0}$. By induction on the length, the claim follows. \\
\item [Step 2)] We  show that the symmetric function $\bar{m}_{(a_1,a_1,c_1)(a_2,a_2,c_2)\cdots(a_l,a_l,c_l)(0,1,0)^k}$ is written as a linear combination  of $\bar{m}_{(\lambda_{0},\cdots,\lambda_{r-1})(0,1,0)^{|\lambda|}}$. It is enough to show that the symmetric function $\bar m_{(a_1,b_1,c_1)\cdots(a_l,b_l,c_l)(0,1,0)^k}$ with $a_i\geq b_i$ for any $i$ is written as a linear combination  of $\bar{m}_{(\lambda_{0},\cdots,\lambda_{r-1})(0,1,0)^{|\lambda|}}$.
We prove the claim by lexicographic induction on $(b,l)$, where $b=\sum_i b_i$. Suppose $b=0$, that is, $b_i=0$ for any $i$. 
If $\sum_ia_i\neq k$, we have
\[\bar m_{(a_1,0,c_1)\cdots(a_l,0,c_l)(0,1,0)^k}=0 \]
by the definition of $J$. 

If  $\sum_ia_i=k$, we define the $r$-tuple of partitions $\lambda=(\lambda_0,\cdots,\lambda_{r-1})$ by $\alpha_{i,j}=\sharp\{t\in\{1,2,\cdots l\}\mid a_t=j, c_t=i\}$. Then we have 
\[\bar m_{(a_1,0,c_1)\cdots(a_l,0,c_l)(0,1,0)^k}=\bar{m}_{(\lambda_{0},\cdots,\lambda_{r-1})(0,1,0)^{|\lambda|}}. \]
Suppose $b>0$.  We may assume that $b_1>0$. In the case of $l=1$, by Lemma \ref{lem1}, we have
\[0=\bar m_{(a_1,b_1-1, c_1)}\bar m_{(0,1,0)^{k+1}}=\bar m_{(a_1,b_1-1,c_1)(0,1,0)^{k+1}}+\bar m_{(a_1,b_1,c_1)(0,1,0)^k}. \]
By induction on $b$,  $\bar m_{(a_1,b_1,c_1)(0,1,0)^k}$ is written as a
 linear combination of $\bar{m}_{(\lambda_{0},\cdots,\lambda_{r-1})(0,1,0)^{|\lambda|}}$. \\
Suppose $l>1$. By Lemma \ref{lem1}, we have
\begin{align}
0=&\bar m_{(a_1,b_1-1,c_1)}\bar m_{(a_2,b_2,c_2)\cdots(a_l,b_l,c_l)(0,1,0)^{k+1}} \notag \\
=&u_0\bar m_{(a_1,b_1-1,c_1)(a_2,b_2,c_2)\cdots(a_l,b_l,c_l)(0,1,0)^{k+1}} +u_1 \bar m_{(a_1,b_1,c_1)(a_2,b_2,c_2)\cdots(a_l,b_l,c_l)(0,1,0)^{k}} \notag \\
&+\sum_{c_1+c_i<r}u_i\bar m_{(a_2,b_2,c_2)\cdots(a_1+a_i,b_1+b_i-1,c_1+c_i)\cdots(a_l,b_l,c_l)(0,1,0)^{k+1}} \notag \\
&+\sum_{c_1+c_i>r-1}u_i\bar m_{(a_2,b_2,c_2)\cdots(a_1+a_i+1,b_1+b_i,c_1+c_i-r)\cdots(a_l,b_l,c_l)(0,1,0)^{k+1}} \notag
\end{align}
for some $u_1>0$, $u_i\in \mathbb{Z}_{\geq 0}$. By the induction hypothesis on $b$, the first term and the third term are written as a linear combination of $\bar{m}_{(\lambda_{0},\cdots,\lambda_{r-1})(0,1,0)^{|\lambda|}}$. The fourth term has the same value of $b$ but has length $l-1$, so it is covered by the secondary induction on $l$. Hence the claim follows.  
\item [Step 3)] Finally, we prove the linear independence of $\{\bar{m}_{(\lambda_{0},\cdots,\lambda_{r-1})(0,1,0)^{|\lambda|}}\}$. Since $S/J\simeq \mathbb{C}[\bar{m}_{(a,a,c)}]$ and $\mathrm{deg}\;\bar{m}_{(a,a,c)}=2(ra+c)$, the degree of the monomial $\bar{m}_{(a_1,a_1,c_1)}\bar{m}_{(a_2,a_2,c_2)}\cdots\bar{m}_{(a_l,a_l,c_l)}$ is equal to  $2\sum_i(r|\lambda_i|+i\,l(\lambda_i))$, where we set the partition $\lambda_i=(0^{\alpha_{i,0}},1^{\alpha_{i,1}},2^{\alpha_{i,2}},\cdots)$ given by  $\alpha_{i,j}=\sharp\{t\in\{1,2,\cdots l\}\mid a_t=j, c_t=i\}$. Hence the dimension of the $2k$-th component is $\sharp\{(\lambda_0,\cdots,\lambda_{r-1})\mid \sum_i(r|\lambda_i|+il(\lambda_i))=k\}$. 
On the other hand, $\mathrm{deg}\;\bar{m}_{(\lambda_{0},\cdots,\lambda_{r-1})(0,1,0)^{|\lambda|}}= 2\sum_i(r|\lambda_i|+il(\lambda_i))$. Since the spanning set is homogeneous and, in each degree, its cardinality equals the dimension of the corresponding homogeneous component of $\mathbb{C}[\bar m_{(a,a,c)}]_{0\leq c\leq r-1,(a,c)\neq(0,0)}$, it is linearly independent. Hence the proof is complete. 
\end{enumerate}
\end{proof}
\subsection{Basis of coordinate ring of torus fixed points }

The purpose of this subsection is to prove the following theorem. 
\begin{theorem}\label{vec}
The set
\[
\left\{\bar m_{(\lambda_0,\cdots,\lambda_{r-1})(0,1,0)^{|\lambda|}}\ \middle|\ 
 |\lambda|+l(\lambda)\leq n\right\}
\]
is a homogeneous basis of $\mathbb{C}[Z_n^{\mathbb{T}}]$.
\end{theorem}
To prove Theorem \ref{vec}, we prepare several lemmas and propositions. 

For any two $r$-tuples of partitions $\lambda=(\lambda_0,\cdots,\lambda_{r-1})$ and $\mu=(\mu_0,\cdots,\mu_{r-1})$, where $\lambda_i=(\lambda_i^1\leq \lambda_i^2 \leq\cdots)=(0^{\alpha_{i,0}},1^{\alpha_{i,1}},2^{\alpha_{i,2}},\cdots)$ and $\mu_i=(\mu_i^1\leq \mu_i^2\leq\cdots)=(0^{\beta_{i,0}},1^{\beta_{i,1}},2^{\beta_{i,2}},\cdots)$, we define $\mu\leq\lambda$ if and only if $\beta_{i,j}\leq \alpha_{i,j}$ for any $i,j$. Then the non-negative integers $P_c(\lambda,\mu)$ and $Q_c(\lambda,\mu)$ are given by
\[\sum_{i}i(l(\lambda_i)-l(\mu_i))+c=P_c(\lambda,\mu)r+Q_c(\lambda,\mu).\]
\begin{lem}\label{p}
For any two $r$-tuples of partitions $\lambda=(\lambda_0,\cdots,\lambda_{r-1})$ and $\mu=(\mu_0,\cdots,\mu_{r-1})$, and any non-negative integers $j,k$ with $0\leq k \leq r-1$ and $\alpha_{k,j}>0$, we set $\hat{\lambda}^{k,j}=(\lambda_0,\cdots,\lambda_{k-1},\lambda_{k}-j,\lambda_{k+1},\cdots,\lambda_{r-1})$, where $\lambda_k-j=(0^{\alpha_{k,0}}, 1^{\alpha_{k,1}},2^{\alpha_{k,2}},\cdots,j^{\alpha_{k,j}-1},\cdots)$. Assume moreover that $\mu\leq \hat{\lambda}^{k,j}$. Then we have 
\begin{equation}
\begin{cases}
P_c(\lambda,\mu)=P_{c+k}(\hat{\lambda}^{k,j},\mu),\quad Q_c(\lambda,\mu)=Q_{c+k}(\hat{\lambda}^{k,j},\mu)& c+k \leq r-1 \\
P_c(\lambda,\mu)=P_{c+k-r}(\hat{\lambda}^{k,j},\mu)+1,\quad Q_c(\lambda,\mu)=Q_{c+k-r}(\hat{\lambda}^{k,j},\mu)& c+k \geq r \notag. 
\end{cases}
\end{equation}
\end{lem}
\begin{proof}
Suppose $c+k\leq r-1$. By the definitions of $P$ and $Q$, we have
\[P_c(\lambda,\mu)r+Q_c(\lambda,\mu)=\sum_ii(l(\lambda_i)-l(\mu_i))+c=\sum_ii(l(\hat{\lambda}^{k,j}_i)-l(\mu_i))+c+k=P_{c+k}(\hat{\lambda}^{k,j},\mu)r+Q_{c+k}(\hat{\lambda}^{k,j},\mu).\]
Hence, $P_c(\lambda,\mu)=P_{c+k}(\hat{\lambda}^{k,j},\mu)$, $Q_c(\lambda,\mu)=Q_{c+k}(\hat{\lambda}^{k,j},\mu)$. \\
If $c+k\geq r$, 
by the definitions of $P$ and $Q$, we have the following two relations.  
\[\sum_i i(l(\hat{\lambda}^{k,j}_i)-l(\mu_i))+c+k-r=P_{c+k-r}(\hat{\lambda}^{k,j},\mu)r+Q_{c+k-r}(\hat{\lambda}^{k,j},\mu),\]
\[\sum_i i(l(\lambda_i)-l(\mu_i))+c=P_c(\lambda,\mu)r+Q_c(\lambda,\mu). \]
Since $\sum i(l(\lambda_i)-l(\mu_i))=\sum i(l(\hat{\lambda}^{k,j}_i)-l(\mu_i))+k$, we get
\[r=(P_c(\lambda,\mu)-P_{c+k-r}(\hat{\lambda}^{k,j},\mu))r+(Q_c(\lambda,\mu)-Q_{c+k-r}(\hat{\lambda}^{k,j},\mu)). \]
Hence, we get 
\begin{equation}
P_{c+k-r}(\hat{\lambda}^{k,j},\mu)=P_c(\lambda,\mu)-1 \text{ and } Q_{c+k-r}(\hat{\lambda}^{k,j},\mu)=Q_c(\lambda,\mu). \notag
\end{equation}
\end{proof}
We use the convention that $1/n! =0$ for every negative integer $n<0$.
For a finite family of non-negative integers $(\theta_{i,j})_{i,j}$ with
$L=\sum_{i,j}\theta_{i,j}$, we use the multinomial notation
\[
\binom{L}{(\theta_{i,j})_{i,j}}
=
\frac{L!}{\prod_{i,j}\theta_{i,j}!}.
\]

We also use the Kronecker delta notation
\[
\delta_{(i',j'),(k,j)}=
\begin{cases}
1 & \text{if }(i',j')=(k,j),\\
0 & \text{otherwise}.
\end{cases}
\]
The following elementary identity for multinomial coefficients will be used
repeatedly in the proof of Proposition \ref{prop:reduction}.  A term with a
negative entry is understood to be zero.  If $N=\sum_{p,q}n_{p,q}$, then
\begin{equation}\label{com}
\sum_{i,j}
\binom{N-1}{(n_{p,q}-\delta_{(p,q),(i,j)})_{p,q}}
=
\binom{N}{(n_{p,q})_{p,q}}.
\end{equation}

This is the form used below.

\begin{prop}\label{prop:reduction}
For $a,b,c\in \mathbb{Z}_{\geq0}$ and an $r$-tuple of partitions
$\lambda=(\lambda_0,\cdots,\lambda_{r-1})$, assume that
$a\geq b>0$ and $0\leq c\leq r-1$.  Then
\begin{align*}
&\bar m_{(a,b,c)(\lambda_0,\cdots,\lambda_{r-1})(0,1,0)^{|\lambda|+a-b}} \notag\\
=&\sum_{\mu\leq\lambda}
D(a,b,c,\lambda,\mu) \notag\\
&\qquad\times
\bar m_{(|\lambda|-|\mu|+a+P_c(\lambda,\mu),0,Q_c(\lambda,\mu))
(\mu_0,\cdots,\mu_{r-1})(0,1,0)^{|\lambda|+a+P_c(\lambda,\mu)}} ,
\end{align*}
where
\[
\mu=(\mu_0,\cdots,\mu_{r-1}),
\qquad
\mu_i=(0^{\beta_{i,0}},1^{\beta_{i,1}},2^{\beta_{i,2}},\cdots).
\]
The sum runs over the $r$-tuples satisfying
$l(\lambda)-l(\mu)\leq b+P_c(\lambda,\mu)$.  For every tuple $\mu^q$
appearing in a chain below, write
$(\mu^q)_i=(0^{\beta^q_{i,0}},1^{\beta^q_{i,1}},
2^{\beta^q_{i,2}},\cdots)$.  Then
\begin{align}
&D(a,b,c,\lambda,\mu) \notag\\
=&(-1)^{b+P_c(\lambda,\mu)}
(\beta_{Q_c(\lambda,\mu),|\lambda|-|\mu|+a+P_c(\lambda,\mu)}+1) \notag\\
&\times
\sum_{p=0}^{l(\lambda)-l(\mu)}
\sum_{\lambda\geq\mu^1>\mu^2>\cdots>\mu^p>\mu^{p+1}=\mu}
(-1)^p \notag\\
&\times
\binom{b+P_c(\lambda,\mu^1)}{(b+P_c(\lambda,\mu^1)-(l(\lambda)-l(\mu^1)),\ (\alpha_{i,j}-\beta^1_{i,j})_{i,j}} \notag\\
&\times
\prod_{q=1}^{p}
\binom{P_c(\lambda,\mu^{q+1})-P_c(\lambda,\mu^q)}
{P_c(\lambda,\mu^{q+1})-P_c(\lambda,\mu^q)-(l(\mu^q)-l(\mu^{q+1})),\ (\beta^q_{i,j}-\beta^{q+1}_{i,j})_{i,j}}. \notag
\end{align}
\end{prop}
\begin{proof}
We prove the formula by induction on the lexicographically ordered pair
$(b,l(\lambda))$. All coefficient comparisons below are made for a fixed terminal tuple
$\mu\leq\lambda$.  A term is understood to be zero if one of the involved
multiplicities becomes negative, or if the displayed inequalities for the
partitions are not satisfied.

First consider the case $l(\lambda)=0$.  Repeatedly applying Lemma
\ref{lem1} to the product with the factor $(0,1,0)$ gives
\[
\bar m_{(a,b,c)(0,1,0)^{a-b}}
=(-1)^b\bar m_{(a,0,c)(0,1,0)^a}.
\]
This is exactly the formula above, since then the only possible terminal tuple
is $\mu=\varnothing$ and $D(a,b,c,\varnothing,\varnothing)=(-1)^b$.

Assume now that $l(\lambda)>0$.  We first verify the inductive step for
$b\geq2$, under the induction hypothesis for all smaller lexicographic pairs;
the initial case $b=1$ will be treated afterward.  Since $a\geq b$, the element
$\bar m_{(a,b-1,c)}$ is zero in $S/J$.  Lemma \ref{lem1} gives
\begin{align}
\bar m_{(a,b,c)\lambda(0,1,0)^{|\lambda|+a-b}}
=&-\bar m_{(a,b-1,c)\lambda(0,1,0)^{|\lambda|+a-b+1}} \notag\\
&-\sum_{\substack{\alpha_{k,j}>0\\ c+k\leq r-1}}
\bar m_{(a+j,b-1,c+k)\hat{\lambda}^{k,j}(0,1,0)^{|\lambda|+a-b+1}} \notag\\
&-\sum_{\substack{\alpha_{k,j}>0\\ c+k\geq r}}
\bar m_{(a+j+1,b,c+k-r)\hat{\lambda}^{k,j}(0,1,0)^{|\lambda|+a-b+1}} .
\label{eq:b-reduction}
\end{align}
The first two kinds of terms are handled by the induction hypothesis on $b$,
and the last kind is handled by the induction hypothesis on $l(\lambda)$.
For a fixed terminal tuple $\mu\leq\lambda$, Lemma \ref{p} shows that
all nonzero contributions from these three kinds of terms involve the same
monomial symmetric function
\[
\bar m_{(|\lambda|-|\mu|+a+P_c(\lambda,\mu),0,Q_c(\lambda,\mu))
\mu(0,1,0)^{|\lambda|+a+P_c(\lambda,\mu)}}.
\]
Indeed, for the first kind this follows directly by replacing $b$ with
$b-1$ in the asserted formula.  For the second kind, use
$|\hat\lambda^{k,j}|=|\lambda|-j$ together with
$P_{c+k}(\hat\lambda^{k,j},\mu)=P_c(\lambda,\mu)$ and
$Q_{c+k}(\hat\lambda^{k,j},\mu)=Q_c(\lambda,\mu)$.
For the third kind, use the same equality of sizes together with
$P_{c+k-r}(\hat\lambda^{k,j},\mu)=P_c(\lambda,\mu)-1$ and
$Q_{c+k-r}(\hat\lambda^{k,j},\mu)=Q_c(\lambda,\mu)$;
the extra $1$ in $a+j+1$ cancels the $-1$ in $P$.
Consequently, it suffices to verify the following identity between the
coefficients of this common monomial symmetric function:
\begin{align}
&D(a,b,c,\lambda,\mu) \notag\\
=&-D(a,b-1,c,\lambda,\mu)
-\sum_{\substack{\alpha_{k,j}>0\\ c+k\leq r-1}}
D(a+j,b-1,c+k,\hat{\lambda}^{k,j},\mu) \notag\\
&-\sum_{\substack{\alpha_{k,j}>0\\ c+k\geq r}}
D(a+j+1,b,c+k-r,\hat{\lambda}^{k,j},\mu).
\label{eq:D-recursion}
\end{align}

We prove \eqref{eq:D-recursion}.  The first term is written as
\begin{align}
&-D(a,b-1,c,\lambda,\mu) \notag\\
=&(-1)^{b+P_c(\lambda,\mu)}
(\beta_{Q_c(\lambda,\mu),|\lambda|-|\mu|+a+P_c(\lambda,\mu)}+1) \notag\\
&\times
\sum_{p=0}^{l(\lambda)-l(\mu)}
\sum_{\lambda\geq\mu^1>\mu^2>\cdots>\mu^p>\mu^{p+1}=\mu}
(-1)^p \notag\\
&\times
\binom{b+P_c(\lambda,\mu^1)-1}
{(b+P_c(\lambda,\mu^1)-1-(l(\lambda)-l(\mu^1))),\
(\alpha_{i,j}-\beta^1_{i,j})_{i,j}} \notag\\
&\times
\prod_{q=1}^{p}
\binom{P_c(\lambda,\mu^{q+1})-P_c(\lambda,\mu^q)}
{P_c(\lambda,\mu^{q+1})-P_c(\lambda,\mu^q)
-(l(\mu^q)-l(\mu^{q+1})),\
(\beta^q_{i,j}-\beta^{q+1}_{i,j})_{i,j}} .
\label{eq:first-term-common}
\end{align}

We next rewrite the second and third kinds of terms in the same chain
notation.  In the definition of
$D(a+j,b-1,c+k,\hat\lambda^{k,j},\mu)$, the sequences are originally of the
form
\[
\hat\lambda^{k,j}\geq \mu^1>\mu^2>\cdots>\mu^p>\mu^{p+1}=\mu,
\]
and the range of $p$ is
\[
0\leq p\leq l(\hat\lambda^{k,j})-l(\mu)
= l(\lambda)-l(\mu)-1.
\]
After rewriting these sequences as sequences starting from $\lambda$, we have
$\mu^1\leq \hat\lambda^{k,j}<\lambda$, and hence
\[
\lambda>\mu^1>\mu^2>\cdots>\mu^p>\mu^{p+1}=\mu .
\]
Thus such a chain can exist only if
\[
p+1\leq l(\lambda)-l(\mu).
\]
Consequently, the term with $p=l(\lambda)-l(\mu)$ is an empty sum.  Moreover,
if we use the common notation
\[
\lambda\geq\mu^1>\mu^2>\cdots>\mu^p>\mu^{p+1}=\mu,
\]
then the additional case $\mu^1=\lambda$ gives no contribution, because then
$\alpha_{i,j}-\beta^1_{i,j}=0$ for all $(i,j)$, and the multinomial
coefficient involving
\[
(\alpha_{i',j'}-\beta^1_{i',j'}-\delta_{(i',j'),(k,j)})_{i',j'}
\]
is zero.  Therefore we may use the formal range
$0\leq p\leq l(\lambda)-l(\mu)$ in the common chain notation; the added terms
are all zero.

By the definitions of $P_c$ and $Q_c$, for $c+k\leq r-1$ and
$\nu\leq \hat\lambda^{k,j}$, we have
\[
P_{c+k}(\hat\lambda^{k,j},\nu)=P_c(\lambda,\nu),
\qquad
Q_{c+k}(\hat\lambda^{k,j},\nu)=Q_c(\lambda,\nu).
\]
Hence the second kind of terms becomes
\begin{align}
&-\sum_{\alpha_{k,j}>0,\ c+k\leq r-1}
D(a+j,b-1,c+k,\hat{\lambda}^{k,j},\mu) \notag\\
=&(-1)^{b+P_c(\lambda,\mu)}
(\beta_{Q_c(\lambda,\mu),|\lambda|-|\mu|+a+P_c(\lambda,\mu)}+1) \notag\\
&\times
\sum_{p=0}^{l(\lambda)-l(\mu)}
\sum_{\lambda\geq\mu^1>\mu^2>\cdots>\mu^p>\mu^{p+1}=\mu}
(-1)^p \notag\\
&\times
\sum_{\alpha_{k,j}-\beta^1_{k,j}>0,\ c+k\leq r-1}
\binom{b+P_c(\lambda,\mu^1)-1}
{(b+P_c(\lambda,\mu^1)-(l(\lambda)-l(\mu^1))),\
 (\alpha_{i',j'}-\beta^1_{i',j'}-\delta_{(i',j'),(k,j)})_{i',j'}} \notag\\
&\times
\prod_{q=1}^{p}
\binom{P_c(\lambda,\mu^{q+1})-P_c(\lambda,\mu^q)}
{P_c(\lambda,\mu^{q+1})-P_c(\lambda,\mu^q)
-(l(\mu^q)-l(\mu^{q+1})),\
(\beta^q_{i,j}-\beta^{q+1}_{i,j})_{i,j}} .
\label{eq:second-term-common}
\end{align}

Similarly, by the definitions of $P_c$ and $Q_c$, for $c+k\geq r$ and
$\nu\leq \hat\lambda^{k,j}$, we have
\[
P_{c+k-r}(\hat\lambda^{k,j},\nu)=P_c(\lambda,\nu)-1,
\qquad
Q_{c+k-r}(\hat\lambda^{k,j},\nu)=Q_c(\lambda,\nu).
\]
Also,
\[
|\hat\lambda^{k,j}|-|\mu|+a+j+1+P_{c+k-r}(\hat\lambda^{k,j},\mu)
=
|\lambda|-|\mu|+a+P_c(\lambda,\mu).
\]
Therefore the third kind of terms has the same common form:
\begin{align}
&-\sum_{\alpha_{k,j}>0,\ c+k\geq r}
D(a+j+1,b,c+k-r,\hat{\lambda}^{k,j},\mu) \notag\\
=&(-1)^{b+P_c(\lambda,\mu)}
(\beta_{Q_c(\lambda,\mu),|\lambda|-|\mu|+a+P_c(\lambda,\mu)}+1) \notag\\
&\times
\sum_{p=0}^{l(\lambda)-l(\mu)}
\sum_{\lambda\geq\mu^1>\mu^2>\cdots>\mu^p>\mu^{p+1}=\mu}
(-1)^p \notag\\
&\times
\sum_{\alpha_{k,j}-\beta^1_{k,j}>0,\ c+k\geq r}
\binom{b+P_c(\lambda,\mu^1)-1}
{(b+P_c(\lambda,\mu^1)-(l(\lambda)-l(\mu^1))),\
 (\alpha_{i',j'}-\beta^1_{i',j'}-\delta_{(i',j'),(k,j)})_{i',j'}} \notag\\
&\times
\prod_{q=1}^{p}
\binom{P_c(\lambda,\mu^{q+1})-P_c(\lambda,\mu^q)}
{P_c(\lambda,\mu^{q+1})-P_c(\lambda,\mu^q)
-(l(\mu^q)-l(\mu^{q+1})),\
(\beta^q_{i,j}-\beta^{q+1}_{i,j})_{i,j}} .
\label{eq:third-term-common}
\end{align}

Now fix a chain
\[
\lambda\geq\mu^1>\mu^2>\cdots>\mu^p>\mu^{p+1}=\mu .
\]
The common factor
\[
(-1)^{b+P_c(\lambda,\mu)}
(\beta_{Q_c(\lambda,\mu),|\lambda|-|\mu|+a+P_c(\lambda,\mu)}+1)(-1)^p
\]
and the product over $q$ are the same in
\eqref{eq:first-term-common}, \eqref{eq:second-term-common}, and
\eqref{eq:third-term-common}.  Hence it remains only to compare the first
multinomial coefficient.  The contribution from the first term gives
\begin{align}
\binom{b+P_c(\lambda,\mu^1)-1}
{(b+P_c(\lambda,\mu^1)-1-(l(\lambda)-l(\mu^1))),\
(\alpha_{i,j}-\beta^1_{i,j})_{i,j}} .
\end{align}
The contributions from the second and third kinds of terms together give
\begin{align}
\sum_{\alpha_{k,j}-\beta^1_{k,j}>0}
\binom{b+P_c(\lambda,\mu^1)-1}
{(b+P_c(\lambda,\mu^1)-(l(\lambda)-l(\mu^1))),\
 (\alpha_{i',j'}-\beta^1_{i',j'}-\delta_{(i',j'),(k,j)})_{i',j'}} .
\end{align}
Therefore, by the multinomial identity \eqref{com}, their sum is
\begin{align}
\binom{b+P_c(\lambda,\mu^1)}
{(b+P_c(\lambda,\mu^1)-(l(\lambda)-l(\mu^1))),\
(\alpha_{i,j}-\beta^1_{i,j})_{i,j}} .
\end{align}
This is exactly the first multinomial coefficient appearing in
$D(a,b,c,\lambda,\mu)$.  Hence the right-hand side of
\eqref{eq:D-recursion} is equal to
\begin{align}
&(-1)^{b+P_c(\lambda,\mu)}
(\beta_{Q_c(\lambda,\mu),|\lambda|-|\mu|+a+P_c(\lambda,\mu)}+1) \notag\\
&\times
\sum_{p=0}^{l(\lambda)-l(\mu)}
\sum_{\lambda\geq\mu^1>\mu^2>\cdots>\mu^p>\mu^{p+1}=\mu}
(-1)^p \notag\\
&\times
\binom{b+P_c(\lambda,\mu^1)}
{(b+P_c(\lambda,\mu^1)-(l(\lambda)-l(\mu^1))),\
(\alpha_{i,j}-\beta^1_{i,j})_{i,j}} \notag\\
&\times
\prod_{q=1}^{p}
\binom{P_c(\lambda,\mu^{q+1})-P_c(\lambda,\mu^q)}
{P_c(\lambda,\mu^{q+1})-P_c(\lambda,\mu^q)
-(l(\mu^q)-l(\mu^{q+1})),\
(\beta^q_{i,j}-\beta^{q+1}_{i,j})_{i,j}} ,
\end{align}
which is $D(a,b,c,\lambda,\mu)$.  This proves
\eqref{eq:D-recursion}, and hence establishes the inductive step for
$b\geq2$.

It remains to verify the initial case $b=1$.  Since $a\geq1$, we have
$\bar m_{(a,0,c)}=0$ in $S/J$.  Applying Lemma \ref{lem1} to
$\bar m_{(a,0,c)}\bar m_{\lambda(0,1,0)^{|\lambda|+a}}$ gives
\begin{align}
\bar m_{(a,1,c)\lambda(0,1,0)^{|\lambda|+a-1}}
=&-(\alpha_{c,a}+1)
\bar m_{(a,0,c)\lambda(0,1,0)^{|\lambda|+a}} \notag\\
&-\sum_{\substack{\alpha_{k,j}>0\\ c+k\leq r-1}}
(\alpha_{c+k,a+j}+1)
\bar m_{(a+j,0,c+k)\hat{\lambda}^{k,j}
(0,1,0)^{|\lambda|+a}} \notag\\
&-\sum_{\substack{\alpha_{k,j}>0\\ c+k\geq r}}
\bar m_{(a+j+1,1,c+k-r)\hat{\lambda}^{k,j}
(0,1,0)^{|\lambda|+a}} .
\label{eq:b-one-reduction}
\end{align}
The first term is the contribution indexed by $\mu=\lambda$, and the second
sum gives the contributions indexed by $\mu=\hat{\lambda}^{k,j}$ with
$c+k\leq r-1$.  For $\mu=\lambda$, only $p=0$ occurs, and
$P_c(\lambda,\lambda)=0$ and $Q_c(\lambda,\lambda)=c$.  Thus
\[
D(a,1,c,\lambda,\lambda)
=-(\alpha_{c,a}+1)\binom{1}{1,(0)_{i',j'}}
=-(\alpha_{c,a}+1).
\]
Next fix $(k,j)$ with $\alpha_{k,j}>0$ and $c+k\leq r-1$, and set
$\mu=\hat\lambda^{k,j}$.  Then
\[
\begin{gathered}
l(\lambda)-l(\mu)=1,\qquad |\lambda|-|\mu|=j,\\
P_c(\lambda,\mu)=0,\qquad Q_c(\lambda,\mu)=c+k,\\
\alpha_{i',j'}-\beta_{i',j'}=\delta_{(i',j'),(k,j)}.
\end{gathered}
\]
Moreover, $a\geq1$ implies $(c+k,a+j)\neq(k,j)$, so
$\beta_{c+k,a+j}=\alpha_{c+k,a+j}$.
The definition of $D$ now has two possible values of $p$.
For $p=0$, the only chain has $\mu^1=\mu$, and its contribution is
\[
-(\alpha_{c+k,a+j}+1)
\binom{1}{0,(\delta_{(i',j'),(k,j)})_{i',j'}}
=-(\alpha_{c+k,a+j}+1).
\]
For $p=1$, the only chain is $\mu^1=\lambda>\mu^2=\mu$.
Since $P_c(\lambda,\mu^2)-P_c(\lambda,\mu^1)=0$, its contribution is
\[
(\alpha_{c+k,a+j}+1)
\binom{1}{1,(0)_{i',j'}}
\binom{0}{-1,(\delta_{(i',j'),(k,j)})_{i',j'}}=0.
\]
Here the last multinomial coefficient is zero because its first lower
entry is negative.  Adding the contributions for $p=0$ and $p=1$ gives
\[
D(a,1,c,\lambda,\hat{\lambda}^{k,j})
=-(\alpha_{c+k,a+j}+1).
\]

We now describe explicitly how the induction hypothesis applies to the last
sum.  Fix $(k,j)$ with $\alpha_{k,j}>0$ and $c+k\geq r$.  Since
\[
l(\hat{\lambda}^{k,j})=l(\lambda)-1,
\qquad
|\hat{\lambda}^{k,j}|+(a+j+1)-1=|\lambda|+a,
\]
and since $a+j+1\geq1$ and $0\leq c+k-r\leq r-1$,
the induction hypothesis for the pair
$(1,l(\hat{\lambda}^{k,j}))$ gives
\begin{align}
&\bar m_{(a+j+1,1,c+k-r)\hat{\lambda}^{k,j}
(0,1,0)^{|\lambda|+a}} \notag\\
=&\sum_{\mu\leq\hat{\lambda}^{k,j}}
D(a+j+1,1,c+k-r,\hat{\lambda}^{k,j},\mu) \notag\\
&\quad\times
\bar m_{(|\hat{\lambda}^{k,j}|-|\mu|+a+j+1
+P_{c+k-r}(\hat{\lambda}^{k,j},\mu),0,
Q_{c+k-r}(\hat{\lambda}^{k,j},\mu))
\mu(0,1,0)^{|\hat{\lambda}^{k,j}|+a+j+1
+P_{c+k-r}(\hat{\lambda}^{k,j},\mu)}} .
\label{eq:b-one-hat-expansion}
\end{align}
By Lemma \ref{p},
\[
P_{c+k-r}(\hat{\lambda}^{k,j},\mu)=P_c(\lambda,\mu)-1,
\qquad
Q_{c+k-r}(\hat{\lambda}^{k,j},\mu)=Q_c(\lambda,\mu).
\]
Hence every term in \eqref{eq:b-one-hat-expansion} has the common basis form
appearing in the statement.  For a fixed terminal tuple $\mu$, the last sum
in \eqref{eq:b-one-reduction} therefore contributes
\begin{align*}
&-\sum_{\substack{\alpha_{k,j}>0,\ c+k\geq r\\
\mu\leq\hat{\lambda}^{k,j}}}
D(a+j+1,1,c+k-r,\hat{\lambda}^{k,j},\mu)\\
&\qquad\times
\bar m_{(|\lambda|-|\mu|+a+P_c(\lambda,\mu),0,Q_c(\lambda,\mu))
\mu(0,1,0)^{|\lambda|+a+P_c(\lambda,\mu)}}.
\end{align*}
The rewriting of these terms in the common chain notation is exactly the
specialization $b=1$ of the calculation in
\eqref{eq:third-term-common}.  The extension of the range of $p$ and the
additional case $\mu^1=\lambda$ are valid for the reasons already given
before \eqref{eq:second-term-common}.  Combining this common-chain expression
with the two boundary contributions in \eqref{eq:b-one-reduction}, and applying
the same multinomial identity \eqref{com}, gives
$D(a,1,c,\lambda,\mu)$.  Thus the assertion holds for $b=1$ as well.  Together
with the inductive step above, this completes the induction.

\end{proof}
Set
\[
f_{\nu}^{\mu}(x)
=
\frac{(x-|\nu|)!(x-|\mu|+1)}
{(x-|\nu|-l(\nu)+l(\mu)+1)!\prod_{i,j}(\gamma_{i,j}-\beta_{i,j})!}.
\]
Here
\[
\mu=(\mu_0,\cdots,\mu_{r-1}),\qquad
\nu=(\nu_0,\cdots,\nu_{r-1}),
\]
where
\[
\mu_i=(0^{\beta_{i,0}},1^{\beta_{i,1}},2^{\beta_{i,2}},\cdots),
\qquad
\nu_i=(0^{\gamma_{i,0}},1^{\gamma_{i,1}},2^{\gamma_{i,2}},\cdots).
\]
The following two lemmata are the direct analogues of \cite[Lemma 2.7, Lemma 2.8]{hikita}.  We use them only as combinatorial identities for the multiplicities of parts.  In the present notation, the single index used in \cite{hikita} is replaced by the pair $(i,j)$, where $0\leq i\leq r-1$ and $j\geq0$, with the pair $(0,0)$ omitted.  Thus zero parts in $\lambda_i$ for $i>0$ are treated in exactly the same way as ordinary parts, while zero parts in $\lambda_0$ are excluded by our convention.  With this replacement of the index set, the proofs in \cite{hikita} apply verbatim.  We record the identities for later use.
\begin{lem}\label{2}
For $\mu=(\mu_0,\cdots,\mu_{r-1})$ and
$\nu=(\nu_0,\cdots,\nu_{r-1})$, we have
\[
f^{\mu}_{\nu}(x+1)-f^{\mu}_{\nu}(x)
=
\sum_{\substack{0\leq i\leq r-1\\ j\in \mathbb{Z}_{\geq 0}}}
f^{(\mu_0,\cdots,\mu_i\cup j,\cdots,\mu_{r-1})}_{\nu}(x).
\]
Here $\mu_i\cup j$ denotes the partition obtained from $\mu_i$ by increasing the multiplicity of the part $j$ by one. If $i=0$, the summand with $j=0$ is omitted, in accordance with the convention $\alpha_{0,0}=0$.
\end{lem}
\begin{lem}\label{3}
Let $\lambda=(\lambda_0,\cdots,\lambda_{r-1})$ and
$\mu=(\mu_0,\cdots,\mu_{r-1})$ be $r$-tuples of partitions such that
$\mu\leq\lambda$, where
\[
\lambda_i=(0^{\alpha_{i,0}},1^{\alpha_{i,1}},2^{\alpha_{i,2}},\cdots),
\qquad
\mu_i=(0^{\beta_{i,0}},1^{\beta_{i,1}},2^{\beta_{i,2}},\cdots).
\]
Then, for every $k\in\mathbb Z_{\geq0}$, we have
\begin{align}
&\sum_{\mu\leq \nu\leq \lambda}
(-1)^{l(\nu)+l(\lambda)}
 f^{\mu}_\nu(k+ |\lambda|) \notag\\
&\qquad\times
\binom{(|\lambda|+l(\lambda))-(|\nu|+l(\nu))}
{\bigl(|\lambda|-|\nu|,\,
(\alpha_{i,j}-\gamma_{i,j})_{i,j}\bigr)} \notag\\
&=
\binom{k}
{\bigl(k-(l(\lambda)-l(\mu)),\,
 (\alpha_{i,j}-\beta_{i,j})_{i,j}\bigr)} , \notag
\end{align}
where
\[
\nu_i=(0^{\gamma_{i,0}},1^{\gamma_{i,1}},
2^{\gamma_{i,2}},\cdots)
\qquad (0\leq i\leq r-1).
\]
\end{lem}

In the sequel, any expression containing $(0,1,0)^m$ with $m<0$ is understood to be zero. Equivalently, the corresponding merging term is absent. If the leading triple
\[
(|\lambda|-|\mu|+a+P_c(\lambda,\mu),0,Q_c(\lambda,\mu))
\]
is equal to $(0,0,0)$ in Lemma \ref{4}, we omit this factor from the corresponding monomial symmetric function.

\begin{lem}\label{4}
For non-negative integers $a,b$, $0\leq c\leq r-1$, and an $r$-tuple of
partitions $\lambda=(\lambda_0,\cdots,\lambda_{r-1})$ with
$|\lambda|+a-b\geq0$, assume moreover that $b>0$ and $(a,b,c)\neq(0,1,0)$. Then we have
\begin{align}
&\bar m_{(a,b,c)(\lambda_0,\cdots,\lambda_{r-1})(0,1,0)^{|\lambda|+a-b}} \notag \\
=&\sum_{\mu\leq \lambda}
(\beta_{Q_c(\lambda,\mu),|\lambda|-|\mu|+a+P_c(\lambda,\mu)}+1) \notag\\
&\quad\times
\bar m_{(|\lambda|-|\mu|+a+P_c(\lambda,\mu),0,Q_c(\lambda,\mu))
(\mu_0,\cdots,\mu_{r-1})(0,1,0)^{|\lambda|+a+P_c(\lambda,\mu)}} \notag\\
&\quad\times
\bigg\{\sum_{p=0}^{l(\lambda)-l(\mu)}
\sum_{\lambda\geq\mu^1>\mu^2>\cdots>\mu^p>\mu^{p+1}=\mu}
\Bigg(
\sum_{\substack{\mu^1\leq \nu \leq \lambda\\
|\nu|\leq |\lambda|+a-b}} \notag\\
&\quad\quad\times
(-1)^{l(\nu)+l(\lambda)+b+P_c(\lambda,\mu)+p}
 f^{\mu^1}_{\nu}(a+P_c(\lambda,\mu^1)+|\lambda|) \notag\\
&\quad\quad\times
\binom{l(\lambda)+|\lambda|-l(\nu)-|\nu|+a-b}
{\bigl(|\lambda|-|\nu|+a-b,\,
(\alpha_{i,j}-\gamma_{i,j})_{i,j}\bigr)}
\Bigg) \notag\\
&\quad\quad\times
\prod_{q=1}^{p}
\binom{P_c(\lambda,\mu^{q+1})-P_c(\lambda,\mu^q)}
{\bigl(P_c(\lambda,\mu^{q+1})-P_c(\lambda,\mu^q)
-(l(\mu^q)-l(\mu^{q+1})),\,
(\beta^q_{i,j}-\beta^{q+1}_{i,j})_{i,j}\bigr)}
 \bigg\}, \notag
\end{align}
where
\[
\mu_i=(0^{\beta_{i,0}},1^{\beta_{i,1}},2^{\beta_{i,2}},\cdots),
\quad
\nu_i=(0^{\gamma_{i,0}},1^{\gamma_{i,1}},2^{\gamma_{i,2}},\cdots),
\]
and, for every tuple \(\mu^q\) occurring in a chain,
\[
(\mu^q)_i=(0^{\beta^q_{i,0}},1^{\beta^q_{i,1}},
2^{\beta^q_{i,2}},\cdots).
\]
In particular, the coefficient of a term indexed by $\mu$ is zero unless
\[
  l(\lambda)-l(\mu)\leq b+P_c(\lambda,\mu).
\]
\end{lem}
\begin{proof}
We prove the formula by induction on
\[
N=|\lambda|+l(\lambda)+a-b.
\]
The case $N=0$ is the case $\lambda=\varnothing$ and $a=b$, and it follows from
\[
\bar m_{(a,a,c)}=(-1)^a\bar m_{(a,0,c)(0,1,0)^a}.
\]
Assume $N>0$.

First suppose $a=b$.  Since $b>0$ by assumption, we have $a=b>0$.
Proposition \ref{prop:reduction} gives
\begin{align}
&\bar m_{(a,a,c)(\lambda_0,\ldots,\lambda_{r-1})(0,1,0)^{|\lambda|}}
\notag\\
={}&\sum_{\mu\leq\lambda}
(\beta_{Q_c(\lambda,\mu),|\lambda|-|\mu|+a+P_c(\lambda,\mu)}+1)
\notag\\
&\quad\times
\bar m_{(|\lambda|-|\mu|+a+P_c(\lambda,\mu),0,Q_c(\lambda,\mu))
\mu(0,1,0)^{|\lambda|+a+P_c(\lambda,\mu)}}
\notag\\
&\quad\times(-1)^{a+P_c(\lambda,\mu)}
\sum_{p=0}^{l(\lambda)-l(\mu)}
\sum_{\lambda\geq\mu^1>\cdots>\mu^p>\mu^{p+1}=\mu}
(-1)^p
\notag\\
&\quad\times
\binom{a+P_c(\lambda,\mu^1)}
{\bigl(a+P_c(\lambda,\mu^1)-(l(\lambda)-l(\mu^1)),\,
(\alpha_{i,j}-\beta^1_{i,j})_{i,j}\bigr)}
\notag\\
&\quad\times
\prod_{q=1}^{p}
\binom{P_c(\lambda,\mu^{q+1})-P_c(\lambda,\mu^q)}
{\bigl(P_c(\lambda,\mu^{q+1})-P_c(\lambda,\mu^q)
-(l(\mu^q)-l(\mu^{q+1})),\,
(\beta^q_{i,j}-\beta^{q+1}_{i,j})_{i,j}\bigr)}.
\label{eq:a-equals-b-D}
\end{align}
For each fixed first element $\mu^1$, Lemma \ref{3}, applied with
$k=a+P_c(\lambda,\mu^1)$, yields
\begin{align}
&\binom{a+P_c(\lambda,\mu^1)}
{\bigl(a+P_c(\lambda,\mu^1)-(l(\lambda)-l(\mu^1)),\,
(\alpha_{i,j}-\beta^1_{i,j})_{i,j}\bigr)}
\notag\\
={}&\sum_{\mu^1\leq\nu\leq\lambda}
(-1)^{l(\nu)+l(\lambda)}
f^{\mu^1}_{\nu}(a+P_c(\lambda,\mu^1)+|\lambda|)
\notag\\
&\quad\times
\binom{l(\lambda)+|\lambda|-l(\nu)-|\nu|}
{\bigl(|\lambda|-|\nu|,\,
(\alpha_{i,j}-\gamma_{i,j})_{i,j}\bigr)}.
\label{eq:a-equals-b-lemma3}
\end{align}
Since $a=b$ and $\nu\leq\lambda$, the bound
$|\nu|\leq|\lambda|+a-b$ is automatic.  Moreover,
\[
(-1)^{a+P_c(\lambda,\mu)+p}(-1)^{l(\nu)+l(\lambda)}
=(-1)^{l(\nu)+l(\lambda)+b+P_c(\lambda,\mu)+p}.
\]
Substituting \eqref{eq:a-equals-b-lemma3} into
\eqref{eq:a-equals-b-D} gives precisely the displayed formula.

It remains to treat the case $a\neq b$.  By the additional assumption of the
lemma, we have $b>0$ and $(a,b,c)\neq(0,1,0)$.  Hence the factor $(a,b,c)$
cannot occur among the factors of $\lambda(0,1,0)^{|\lambda|+a-b}$: the
factors of $\lambda$ have middle coordinate zero, and the only factor with
non-zero middle coordinate among the added factors is $(0,1,0)$.  Applying Lemma \ref{lem1} to
$\bar m_{(a,b,c)}\bar m_{\lambda(0,1,0)^{|\lambda|+a-b}}=0$ in $S/J$, we get
\begin{align}
&\bar m_{(a,b,c)(\lambda_0,\cdots,\lambda_{r-1})(0,1,0)^{|\lambda|+a-b}} \notag\\
={}&-\bar m_{(a,b+1,c)(\lambda_0,\cdots,\lambda_{r-1})(0,1,0)^{|\lambda|+a-(b+1)}} \notag \\
&-\sum_{\substack{\alpha_{k,j}>0 \\ k+c\leq r-1}}
\bar m_{(a+j,b,c+k)\hat\lambda^{k,j}(0,1,0)^{|\lambda|+a-b}} \notag \\
&-\sum_{\substack{\alpha_{k,j}>0 \\ k+c\geq r}}
\bar m_{(a+j+1,b+1,c+k-r)\hat\lambda^{k,j}(0,1,0)^{|\lambda|+a-(b+1)}}. \notag
\end{align}
Each term on the right has smaller value of $N$: the first term lowers
$a-b$ by one, and each term containing $\hat\lambda^{k,j}$ lowers the length by
one.  Therefore the induction hypothesis applies to all terms on the right.
Using Lemma \ref{p}, the terms containing $\hat\lambda^{k,j}$ are rewritten in
the common form appearing in the statement. More precisely, the first term on
the right contributes the part in which no factor of $\lambda$ is removed at the
first step. The second and third sums contribute the parts in which the first
removed factor is $(j,0,k)$. After the $P$- and $Q$-indices are rewritten by
Lemma \ref{p}, these contributions are indexed by the same sequences
\[
\lambda\geq\mu^1>\mu^2>\cdots>\mu^p>\mu^{p+1}=\mu .
\]
We now write the coefficient calculation explicitly.  Fix a terminal tuple
$\mu$, a chain
\[
\lambda\geq\mu^1>\mu^2>\cdots>\mu^p>\mu^{p+1}=\mu,
\]
and a tuple $\nu$ with $\mu^1\leq\nu\leq\lambda$ and
$|\nu|\leq|\lambda|+a-b$.  The negative of the first term on the right has
the required common sign, since
\[
-(-1)^{l(\nu)+l(\lambda)+(b+1)+P_c(\lambda,\mu)+p}
=(-1)^{l(\nu)+l(\lambda)+b+P_c(\lambda,\mu)+p}.
\]
Its remaining multinomial factor is
\begin{equation}
\binom{l(\lambda)+|\lambda|-l(\nu)-|\nu|+a-b-1}
{\bigl(|\lambda|-|\nu|+a-b-1,\,
(\alpha_{i,j}-\gamma_{i,j})_{i,j}\bigr)}.
\label{eq:lem4-first-multinomial}
\end{equation}
Indeed, its exponent of $(0,1,0)$ is smaller by one.  Its original range is
$|\nu|\leq|\lambda|+a-b-1$.  We may enlarge this to
$|\nu|\leq|\lambda|+a-b$, since the additional case makes the first lower
entry in \eqref{eq:lem4-first-multinomial} equal to $-1$, and hence gives zero.

For a removed factor $(j,0,k)$, the condition $\nu\leq\hat\lambda^{k,j}$
is equivalent to $\alpha_{k,j}-\gamma_{k,j}>0$.  If $c+k\leq r-1$, then,
for $1\leq q\leq p+1$,
\begin{align*}
P_{c+k}(\hat\lambda^{k,j},\mu^q)&=P_c(\lambda,\mu^q),&
Q_{c+k}(\hat\lambda^{k,j},\mu^q)&=Q_c(\lambda,\mu^q),\\
|\hat\lambda^{k,j}|+a+j&=|\lambda|+a,&
l(\hat\lambda^{k,j})&=l(\lambda)-1,
\end{align*}
whereas, if $c+k\geq r$, then, for $1\leq q\leq p+1$,
\begin{align*}
P_{c+k-r}(\hat\lambda^{k,j},\mu^q)&=P_c(\lambda,\mu^q)-1,&
Q_{c+k-r}(\hat\lambda^{k,j},\mu^q)&=Q_c(\lambda,\mu^q),\\
|\hat\lambda^{k,j}|+a+j+1&=|\lambda|+a+1,&
l(\hat\lambda^{k,j})&=l(\lambda)-1.
\end{align*}
Since $\mu^{p+1}=\mu$, these identities show directly that the argument of
$f^{\mu^1}_{\nu}$, the factor indexed by $\mu$, and every multinomial in the
product over $q$ agree with those in the asserted formula.  The signs agree as
well: for the non-wrapping and wrapping terms, respectively,
\begin{align*}
&-(-1)^{l(\nu)+l(\hat\lambda^{k,j})+b
+P_{c+k}(\hat\lambda^{k,j},\mu)+p}
=(-1)^{l(\nu)+l(\lambda)+b+P_c(\lambda,\mu)+p},\\
&-(-1)^{l(\nu)+l(\hat\lambda^{k,j})+(b+1)
+P_{c+k-r}(\hat\lambda^{k,j},\mu)+p}
=(-1)^{l(\nu)+l(\lambda)+b+P_c(\lambda,\mu)+p}.
\end{align*}
Consequently, the two sums over removed factors together contribute
\begin{equation}
\sum_{\alpha_{k,j}-\gamma_{k,j}>0}
\binom{l(\lambda)+|\lambda|-l(\nu)-|\nu|+a-b-1}
{\bigl(|\lambda|-|\nu|+a-b,\,
 (\alpha_{i',j'}-\gamma_{i',j'}
 -\delta_{(i',j'),(k,j)})_{i',j'}\bigr)}.
\label{eq:lem4-removed-multinomials}
\end{equation}
All factors common to \eqref{eq:lem4-first-multinomial} and
\eqref{eq:lem4-removed-multinomials} have been suppressed.  Applying
\eqref{com} to the family consisting of $|\lambda|-|\nu|+a-b$ and
the $\alpha_{i,j}-\gamma_{i,j}$ gives
\begin{align*}
&\binom{l(\lambda)+|\lambda|-l(\nu)-|\nu|+a-b-1}
{\bigl(|\lambda|-|\nu|+a-b-1,\,
(\alpha_{i,j}-\gamma_{i,j})_{i,j}\bigr)} \notag\\
&\quad+
\sum_{\alpha_{k,j}-\gamma_{k,j}>0}
\binom{l(\lambda)+|\lambda|-l(\nu)-|\nu|+a-b-1}
{\bigl(|\lambda|-|\nu|+a-b,\,
 (\alpha_{i',j'}-\gamma_{i',j'}
 -\delta_{(i',j'),(k,j)})_{i',j'}\bigr)}\\
&\qquad=
\binom{l(\lambda)+|\lambda|-l(\nu)-|\nu|+a-b}
{\bigl(|\lambda|-|\nu|+a-b,\,
(\alpha_{i,j}-\gamma_{i,j})_{i,j}\bigr)}.
\end{align*}
This is the multinomial factor in the displayed coefficient, so the induction
is complete.  The final vanishing condition follows immediately from the
displayed binomial factor in the coefficient.
\end{proof}

The exceptional case $(a,b,c)=(0,0,c)$ with $c>0$ is not included in Lemma \ref{4}.  Indeed, in this case no reduction is necessary.  For an $r$-tuple $\rho$ and $c>0$, let $\rho\cup_c 0$ denote the $r$-tuple obtained from $\rho$ by adding one zero part to $\rho_c$.  Then
\[
\bar m_{(0,0,c)\rho(0,1,0)^{|\rho|}}
=
\bar m_{(\rho\cup_c 0)(0,1,0)^{|\rho\cup_c 0|}}.
\]
Here $|\rho\cup_c 0|=|\rho|$ and $l(\rho\cup_c 0)=l(\rho)+1$.  We will use this direct formula for the zero case below.

For later use, we introduce the following notation. Let
$\eta=(\eta_0,\cdots,\eta_{r-1})$ be an $r$-tuple of partitions. For
non-negative integers $A,B,C$ with $0\leq C\leq r-1$, we define
$\mathcal{E}^{\eta}_{A,B,C}(\mu)$ as follows. If $|\eta|+A-B\geq0$, $B>0$, and $(A,B,C)\neq(0,1,0)$, then
$\mathcal{E}^{\eta}_{A,B,C}(\mu)$ denotes the coefficient of the basis vector
\[
\bar m_{(|\eta|-|\mu|+A+P_C(\eta,\mu),0,Q_C(\eta,\mu))
\mu(0,1,0)^{|\eta|+A+P_C(\eta,\mu)}}
\]
in the expansion of
\[
\bar m_{(A,B,C)\eta(0,1,0)^{|\eta|+A-B}}
\]
given by Lemma \ref{4}. For all other triples $(A,B,C)$, or if
$|\eta|+A-B<0$, we set
\[
\mathcal{E}^{\eta}_{A,B,C}(\mu)=0.
\]
Thus
\[
\mathcal{E}^{\eta}_{A,B,C}(\mu)=0
\quad\text{unless}\quad
l(\eta)-l(\mu)\leq B+P_C(\eta,\mu).
\]

\begin{lem}\label{5}
Assume that $(a,c)\neq(0,0)$. Then
\begin{align}
&\bar m_{(a,a,c)}
\bar m_{(\lambda_0,\cdots,\lambda_{r-1})(0,1,0)^{|\lambda|}} \notag\\
={}&
\varepsilon_a^+
\sum_{\mu\leq\lambda}
\mathcal{E}^{\lambda}_{a,a,c}(\mu)
\bar m_{(|\lambda|-|\mu|+a+P_c(\lambda,\mu),0,Q_c(\lambda,\mu))
\mu(0,1,0)^{|\lambda|+a+P_c(\lambda,\mu)}} \notag\\
&+
\varepsilon_a^0(1+\alpha_{c,0})
\bar m_{(\lambda\cup_c 0)(0,1,0)^{|\lambda\cup_c 0|}} \notag\\
&+
\sum_{\mu\leq\lambda}
\mathcal{E}^{\lambda}_{a,a+1,c}(\mu)
\bar m_{(|\lambda|-|\mu|+a+P_c(\lambda,\mu),0,Q_c(\lambda,\mu))
\mu(0,1,0)^{|\lambda|+a+P_c(\lambda,\mu)}} \notag\\
&+
\varepsilon_a^+
\sum_{\substack{\alpha_{k,j}>0\\ c+k\leq r-1}}
\sum_{\mu\leq\hat\lambda^{k,j}}
\mathcal{E}^{\hat\lambda^{k,j}}_{a+j,a,c+k}(\mu) \notag\\
&\qquad\qquad\qquad\qquad\times
\bar m_{(|\lambda|-|\mu|+a+P_c(\lambda,\mu),0,Q_c(\lambda,\mu))
\mu(0,1,0)^{|\lambda|+a+P_c(\lambda,\mu)}} \notag\\
&+
\varepsilon_a^0
\sum_{\substack{\alpha_{k,j}>0\\ c+k\leq r-1}}
(1+\alpha_{c+k,j})
\bar m_{(j,0,c+k)\hat\lambda^{k,j}(0,1,0)^{|\lambda|}} \notag\\
&+
\sum_{\substack{\alpha_{k,j}>0\\ c+k\geq r}}
\sum_{\mu\leq\hat\lambda^{k,j}}
\mathcal{E}^{\hat\lambda^{k,j}}_{a+j+1,a+1,c+k-r}(\mu) \notag\\
&\qquad\qquad\qquad\qquad\times
\bar m_{(|\lambda|-|\mu|+a+P_c(\lambda,\mu),0,Q_c(\lambda,\mu))
\mu(0,1,0)^{|\lambda|+a+P_c(\lambda,\mu)}} , \notag
\end{align}
where
\[
\varepsilon_a^+=
\begin{cases}
1 & \text{if } a>0,\\
0 & \text{if } a=0,
\end{cases}
\qquad
\varepsilon_a^0=
\begin{cases}
1 & \text{if } a=0,\\
0 & \text{if } a>0.
\end{cases}
\]
\end{lem}
\begin{proof}
We apply Lemma \ref{lem1} to
$\bar m_{(a,a,c)}\bar m_{\lambda(0,1,0)^{|\lambda|}}$.
The coefficient of the non-merging term is
$1+\delta_{a,0}\alpha_{c,0}$, because the factor $(a,a,c)$ can already occur in
$\lambda(0,1,0)^{|\lambda|}$ only when $a=0$, in which case it occurs as
$(0,0,c)$.  If $a>0$, this non-merging term is covered by Lemma \ref{4}.  If
$a=0$, then $c>0$ by the assumption $(a,c)\neq(0,0)$, and the direct formula
above gives the displayed term
\[
(1+\alpha_{c,0})\bar m_{(\lambda\cup_c 0)(0,1,0)^{|\lambda\cup_c 0|}}.
\]

The merging with a factor $(0,1,0)$ gives
\[
\bar m_{(a,a+1,c)\lambda(0,1,0)^{|\lambda|-1}},
\]
which is interpreted as zero when $|\lambda|=0$.
Next, merging with a factor $(j,0,k)$ of $\lambda$ gives two cases. If
$c+k\leq r-1$, the merged factor is $(a+j,a,c+k)$.  When $a>0$, this term is
covered by Lemma \ref{4}.  When $a=0$, the middle coordinate is zero, and the
term
\[
\bar m_{(j,0,c+k)\hat\lambda^{k,j}(0,1,0)^{|\lambda|}}
\]
is already one of the basis vectors of Lemma \ref{basis}.  Its coefficient
is $1+\alpha_{c+k,j}$.  If $c+k\geq r$, the merged factor is
$(a+j+1,a+1,c+k-r)$, and Lemma \ref{4} applies because the middle coordinate is
$a+1>0$.

Thus the product is the sum of the terms displayed in the
statement.  For the terms to which Lemma \ref{4} is applied, Lemma \ref{p}
rewrites the corresponding $P$- and $Q$-functions as follows.  If $c+k\leq r-1$,
then
\[
P_{c+k}(\hat\lambda^{k,j},\mu)=P_c(\lambda,\mu),\qquad
Q_{c+k}(\hat\lambda^{k,j},\mu)=Q_c(\lambda,\mu).
\]
If $c+k\geq r$, then
\[
P_{c+k-r}(\hat\lambda^{k,j},\mu)+1=P_c(\lambda,\mu),\qquad
Q_{c+k-r}(\hat\lambda^{k,j},\mu)=Q_c(\lambda,\mu).
\]
Together with
$|\hat\lambda^{k,j}|+a+j=|\lambda|+a$ in the former case and
\[
|\hat\lambda^{k,j}|+a+j+1+P_{c+k-r}(\hat\lambda^{k,j},\mu)
=|\lambda|+a+P_c(\lambda,\mu)
\]
in the latter case, all resulting terms take the displayed form.  This proves
the formula.
\end{proof}

\begin{cor}\label{cor}
Assume that $(a,c)\neq(0,0)$. The symmetric function
\[
\bar m_{(a,a,c)}
\bar m_{(\lambda_0,\cdots,\lambda_{r-1})(0,1,0)^{|\lambda|}}
\]
is contained in the subspace spanned by
\[
\left\{
\bar m_{(\nu_0,\cdots,\nu_{r-1})(0,1,0)^{|\nu|}}
\ \middle|\
|\lambda|+l(\lambda)\leq |\nu|+l(\nu)
\right\}.
\]
\end{cor}

\begin{proof}
By Lemma \ref{5}, there are two kinds of terms.  First consider a term involving
one of the coefficients $\mathcal{E}$.  Such a term has the form
\[
\bar m_{(|\lambda|-|\mu|+a+P_c(\lambda,\mu),0,Q_c(\lambda,\mu))
\mu(0,1,0)^{|\lambda|+a+P_c(\lambda,\mu)}}
\]
for some $\mu\leq\lambda$ or $\mu\leq\hat\lambda^{k,j}$.  Define
$\nu=(\nu_0,\cdots,\nu_{r-1})$ by
\[
\nu_i=\mu_i \quad (i\neq Q_c(\lambda,\mu)),
\]
and
\[
\nu_{Q_c(\lambda,\mu)}=
\mu_{Q_c(\lambda,\mu)}\cup\{|\lambda|-|\mu|+a+P_c(\lambda,\mu)\}.
\]
Then the term is equal to
\[
\bar m_{(\nu_0,\cdots,\nu_{r-1})(0,1,0)^{|\nu|}},
\]
and
\[
|\nu|+l(\nu)=|\lambda|+a+P_c(\lambda,\mu)+l(\mu)+1.
\]
For the terms coming from $\mathcal{E}^{\lambda}_{a,a,c}$ and
$\mathcal{E}^{\lambda}_{a,a+1,c}$, Lemma \ref{4} gives respectively
\[
l(\lambda)-l(\mu)\leq a+P_c(\lambda,\mu),
\qquad
l(\lambda)-l(\mu)\leq a+1+P_c(\lambda,\mu).
\]
Hence $|\nu|+l(\nu)\geq |\lambda|+l(\lambda)$.

For the terms containing $\hat\lambda^{k,j}$ and an $\mathcal{E}$-coefficient,
we use $l(\hat\lambda^{k,j})=l(\lambda)-1$.  If $c+k\leq r-1$, Lemma \ref{4}
and Lemma \ref{p} give
\[
l(\hat\lambda^{k,j})-l(\mu)\leq a+P_c(\lambda,\mu),
\]
and hence
$l(\lambda)-l(\mu)\leq a+1+P_c(\lambda,\mu)$.  If $c+k\geq r$, Lemma \ref{p}
gives $P_{c+k-r}(\hat\lambda^{k,j},\mu)+1=P_c(\lambda,\mu)$, and the same
inequality follows again.  Thus all terms involving $\mathcal{E}$ satisfy the
required bound.

It remains to consider the direct terms in Lemma \ref{5}.  First, if $a=0$, the
non-merging direct term is
\[
\bar m_{(\lambda\cup_c 0)(0,1,0)^{|\lambda\cup_c 0|}}.
\]
For the corresponding tuple $\nu=\lambda\cup_c 0$, one has $|\nu|=|\lambda|$ and
$l(\nu)=l(\lambda)+1$, and hence the required inequality is immediate.

The other direct terms occur only when $a=0$ and $c+k\leq r-1$.  Such a term is
\[
\bar m_{(j,0,c+k)\hat\lambda^{k,j}(0,1,0)^{|\lambda|}}.
\]
It is the basis vector corresponding to the tuple obtained from
$\hat\lambda^{k,j}$ by adding the part $j$ to the $(c+k)$-th component.  For
this tuple $\nu$, one has $|\nu|=|\lambda|$ and $l(\nu)=l(\lambda)$.  Hence
$|\nu|+l(\nu)=|\lambda|+l(\lambda)$.
This proves the assertion.
\end{proof}

\begin{proof}[Proof of Theorem \ref{vec}]
If $|\lambda|+l(\lambda)>n$, then the tri-partition $(\lambda_0,\cdots,\lambda_{r-1})(0,1,0)^{|\lambda|}$ has length $|\lambda|+l(\lambda)>n$, and hence the corresponding element is killed in the quotient of Proposition \ref{finitequot}. Thus it remains to show that every element killed by the length condition is contained in the span of such basis vectors.
It is enough to show that every $\bar m_{\nu}$ with $l(\nu)>n$ is contained in the subspace
\begin{equation}
\big\langle \bar m_{(\lambda_0,\cdots,\lambda_{r-1})(0,1,0)^{|\lambda|}}\mid 
|\lambda|+l(\lambda)\geq l(\nu)\big\rangle .
\tag{*}
\end{equation}
Let $e(\nu)$ be the number of factors equal to $(0,1,0)$ in $\nu$, and set
\[
B(\nu)=\sum_{(a,b,c)\in\nu} b-e(\nu),
\qquad
C(\nu)=\sum_{(a,b,c)\in\nu} c.
\]
We prove the assertion by induction on the lexicographically ordered pair $(B(\nu),C(\nu))$. If $B(\nu)=0$, then every factor of $\nu$ other than $(0,1,0)$ is of the form $(a,0,c)$. Hence $\bar m_{\nu}$ is either one of the spanning vectors in $(*)$, or is zero in $S/J$ by the defining relations of $J$. Thus the assertion follows in this case.

Assume $B(\nu)>0$. Then there exists a factor $(a,b,c)$ of $\nu$ such that $b>0$ and $(a,b,c)\neq (0,1,0)$. Write $\nu=\nu'(a,b,c)$. By Lemma \ref{lem1}, we have
\begin{align}
\bar m_{(a,b-1,c)}\bar m_{(0,1,0)\nu'}&=c_0\bar m_{(a,b-1,c)(0,1,0)\nu'}+c_1\bar m_{\nu} \notag \\
&+\sum_{\substack{e(\nu'')=e(\nu)\\ l(\nu'')=l(\nu)-1}}c_{\nu''}\bar m_{(0,1,0)\nu''}, \notag
\end{align}
for some $c_1\neq 0$. The first term $\bar m_{(a,b-1,c)(0,1,0)\nu'}$ has smaller $B$-value than $\bar m_{\nu}$. For a term $\bar m_{(0,1,0)\nu''}$, if the corresponding merging does not wrap around the relation $z^r=xy$, then its $B$-value is smaller than $B(\nu)$. If the merging wraps around, then its $B$-value may be equal to $B(\nu)$, but its $C$-value is smaller by $r$. Hence all these terms have lexicographically smaller pairs $(B,C)$ than $\bar m_{\nu}$, and therefore they belong to $(*)$ by the induction hypothesis.

If $a\neq b-1$, then $\bar m_{(a,b-1,c)}=0$ in $S/J$. Therefore $\bar m_{\nu}$ belongs to $(*)$. If $a=b-1$, then $\bar m_{(a,b-1,c)}=\bar m_{(a,a,c)}$. Since $(a,b,c)\neq(0,1,0)$, the equality $a=b-1$ implies $(a,c)\neq(0,0)$, so Corollary \ref{cor} is applicable. The induction hypothesis gives
\[
\bar m_{(0,1,0)\nu'}\in (*)
\]
because $B((0,1,0)\nu')<B(\nu)$ and $l((0,1,0)\nu')=l(\nu)$. By Corollary \ref{cor}, multiplication by $\bar m_{(a,a,c)}$ does not decrease the quantity $|\lambda|+l(\lambda)$ of the spanning vectors. Hence $\bar m_{\nu}$ is contained in $(*)$.
\end{proof}

\begin{proof}[Proof of Theorem \ref{hatano}]
By Theorem \ref{vec}, the dimension of the degree $2k$ part of $\mathbb{C}[Z_n^{\mathbb{T}}]$ is the cardinality of
\[
A_{n,k}=\left\{\lambda=(\lambda_0,\cdots,\lambda_{r-1})\left|
 |\lambda|+l(\lambda)\leq n,\quad
 \sum_{i=0}^{r-1}\bigl(r|\lambda_i|+i\,l(\lambda_i)\bigr)=k
\right.\right\}.
\]
We compare this set with the indexing set in \cite[Theorem 3.8]{nak}. Let $\lambda\in A_{n,k}$ and put
\[
h=n-(|\lambda|+l(\lambda)).
\]
Add $h$ zero parts to $\lambda_0$ and denote the resulting partition by $\widetilde{\lambda}_0$; set $\widetilde{\lambda}_i=\lambda_i$ for $i>0$. For $0\leq i\leq r-1$, define a Young diagram $Y_{r-i}$ by adding $1$ to every part of $\widetilde{\lambda}_i$. Then
\[
\sum_{\alpha=1}^{r}|Y_\alpha|=n.
\]
Moreover,
\begin{align}
\sum_{\alpha=1}^{r}\bigl(r|Y_\alpha|-\alpha l(Y_\alpha)\bigr)
&=\sum_{i=0}^{r-1}\bigl(r|Y_{r-i}|-(r-i)l(Y_{r-i})\bigr) \notag\\
&=\sum_{i=0}^{r-1}\bigl(r|\lambda_i|+i\,l(\lambda_i)\bigr) \notag\\
&=k. \notag
\end{align}
Conversely, given an $r$-tuple of Young diagrams $(Y_1,\cdots,Y_r)$ with $\sum_\alpha |Y_\alpha|=n$, subtract $1$ from every part of $Y_{r-i}$ to obtain $\lambda_i$ for $i>0$, keeping the resulting zero parts, and subtract $1$ from every part of $Y_r$ to obtain $\lambda_0$, discarding the resulting zero parts. This is inverse to the construction above. Therefore $A_{n,k}$ has the same cardinality as the set of $r$-tuples of Young diagrams appearing in \cite[Theorem 3.8]{nak} with contribution $t^{2k}$. Hence the graded dimensions of $\mathbb{C}[Z_n^{\mathbb{T}}]$ and $H^*(M_{r,n})$ agree in every degree. Since both graded vector spaces have finite-dimensional homogeneous components, this equality of graded dimensions gives a non-canonical isomorphism of graded vector spaces.
\end{proof}

\appendix
\section{A multiplicative relation between the generators}
\label{sec:multiplicative-relation}

In this section, we compare a family of products of generators on the two
sides of Theorem~\ref{hatano}.  This comparison is not needed for the proof of
Theorem~\ref{hatano}.  We include it because it gives a multiplicative
refinement of the graded-vector-space comparison for products lying on the
boundary of the length condition in Theorem~\ref{vec}.

Throughout this section, $r,n\geq1$ and cohomology has complex coefficients.
Let $\mathcal V$ be the tautological vector bundle of rank $n$ on $M_{r,n}$,
and put
\[
 q_d=d!\operatorname{ch}_d(\mathcal V)\in H^{2d}(M_{r,n},\mathbb C).
\]
On the fixed-point side, set
\[
 p_{ra+c}=\bar m_{(a,a,c)}
 \quad(a\geq0,\ 0\leq c<r,\ (a,c)\neq(0,0)).
\]
Thus $p_d$ has degree $2d$.  At finite level it is the image of
$\sum_{i=1}^n(X_iY_i)^d$.  We use the same notation for its images in
$S/J$ and $\mathbb C[Z_n^{\mathbb T}]$.

\begin{prop}\label{prop:pair-product}
Let $a,a'\geq0$ and $0\leq c,c'<r$, and suppose that
$(a,c),(a',c')\neq(0,0)$.  Set
\[
 s=a+a',\qquad
 \epsilon=\left\lfloor\frac{c+c'}r\right\rfloor,
 \qquad d=c+c'-r\epsilon,
 \qquad
 \delta=\delta_{a,a'}\delta_{c,c'}.
\]
Then, in $S/J$, one has
\begin{align}
 p_{ra+c}p_{ra'+c'}
={}&(-1)^s(1+\delta)
 \bar m_{(a,0,c)(a',0,c')(0,1,0)^s}
 +(s+1)p_{r(s+\epsilon)+d}.                         \label{eq:pair-product}
\end{align}
Consequently, in $\mathbb C[Z_n^{\mathbb T}]$, one has
\begin{equation}
 p_{ra+c}p_{ra'+c'}=n p_{r(n-1)+c+c'}               \label{eq:boundary-fixed}
\end{equation}
whenever
\[
 a+a'=n-1\qquad\text{and}\qquad c+c'<r.
\]
\end{prop}

\begin{proof}
For an $r$-tuple consisting of the single part $a'$ in its $c'$-th
component, the only tuples $\mu\leq\lambda$ occurring in
Lemma~\ref{5} are $\lambda$ and the empty tuple.  The corresponding values
of the functions $P_c$ and $Q_c$ are
\[
 P_c(\lambda,\lambda)=0,\quad Q_c(\lambda,\lambda)=c,
 \qquad
 P_c(\lambda,\varnothing)=\epsilon,
 \quad Q_c(\lambda,\varnothing)=d.
\]
We first use the relation
\begin{equation}
 p_{ra'+c'}=(-1)^{a'}
 \bar m_{(a',0,c')(0,1,0)^{a'}}.                    \label{eq:one-part-reduction}
\end{equation}
Specializing Lemma~\ref{5} to this one-part tuple shows that only the two
basis vectors
\[
 B_{\mathrm{sep}}
 =\bar m_{(a,0,c)(a',0,c')(0,1,0)^s},
 \qquad
 B_{\mathrm{mer}}
 =\bar m_{(s+\epsilon,0,d)(0,1,0)^{s+\epsilon}}
\]
can occur.  The term in which the two distinguished factors are not merged
has coefficient $(-1)^a(1+\delta)$ in front of
$B_{\mathrm{sep}}$.  The three contributions to $B_{\mathrm{mer}}$ are
\begin{equation}
 (-1)^{a+\epsilon}a,\qquad
 (-1)^{a+\epsilon}a',\qquad
 (-1)^{a+\epsilon},                                  \label{eq:three-contributions}
\end{equation}
coming respectively from the reduction of the non-merging term, the merger
with one of the $a'$ factors $(0,1,0)$, and the merger with
$(a',0,c')$.  When $a'=0$, the second contribution in
\eqref{eq:three-contributions} is absent and is already equal to zero.  The
direct terms in Lemma~\ref{5} give the same formula when $a=0$.
It follows that
\[
 p_{ra+c}
 \bar m_{(a',0,c')(0,1,0)^{a'}}
 =(-1)^a(1+\delta)B_{\mathrm{sep}}
  +(-1)^{a+\epsilon}(s+1)B_{\mathrm{mer}}.
\]
Multiplying by $(-1)^{a'}$ and using
\[
 p_{r(s+\epsilon)+d}=(-1)^{s+\epsilon}B_{\mathrm{mer}}
\]
gives \eqref{eq:pair-product}.

For completeness, the coefficients can also be checked directly by
inclusion--exclusion over set partitions.  Label the two distinguished
factors $(a,0,c),(a',0,c')$ and the $s$ factors $(0,1,0)$.
In the expansion into products of power sums, a block of size $h$ has
coefficient $(-1)^{h-1}(h-1)!$.  Modulo $J$, only blocks of total
$\mathbb T$-weight zero remain.  If the two distinguished factors lie in
different blocks, those blocks must contain respectively $a$ and $a'$
of the other factors.  Their total contribution, after division by
the normalization factor $s!(1+\delta)$, is
\[
 \frac{\binom{s}{a}(-1)^aa!(-1)^{a'}a'!}{s!(1+\delta)}
 p_{ra+c}p_{ra'+c'}
 =\frac{(-1)^s}{1+\delta}p_{ra+c}p_{ra'+c'}.
\]
If the distinguished factors lie in the same block, all $s+2$ factors
must lie in that block, giving
\[
 \frac{(-1)^{s+1}(s+1)!}{s!(1+\delta)}
 p_{rs+c+c'}
 =\frac{(-1)^{s+1}(s+1)}{1+\delta}p_{rs+c+c'}.
\]
There are no other surviving blocks, since a block containing only
$(0,1,0)$ factors has nonzero weight.  This proves the identity
\[
 B_{\mathrm{sep}}
 =\frac{(-1)^s}{1+\delta}
 \bigl(p_{ra+c}p_{ra'+c'}-(s+1)p_{rs+c+c'}\bigr),
\]
and verifies both the multiplicity factor and the coefficient $s+1$
in \eqref{eq:pair-product}, including the cases $a=0$ or $a'=0$.

Now assume $s=n-1$ and $c+c'<r$.  Then $\epsilon=0$, and
$B_{\mathrm{sep}}$ has length
\[
 s+2=n+1.
\]
It is therefore killed by the finite-level quotient of
Proposition~\ref{finitequot}.  Since $s+1=n$, formula
\eqref{eq:boundary-fixed} follows.
\end{proof}

More generally, the lengths of the two terms in
\eqref{eq:pair-product}, written in the basis of Lemma~\ref{basis},
are $s+2$ and $s+\epsilon+1$, respectively.  Thus both terms survive
when $s\leq n-2$, and both vanish when $s\geq n$ or when
$s=n-1$ and $c+c'\geq r$.  In the latter vanishing cases the total
degree is at least $2rn$.  The same vanishing holds in cohomology,
since Theorem~\ref{hatano} and the degree formula for the basis of
Theorem~\ref{vec} give
$H^{2d}(M_{r,n},\mathbb C)=0$ for $d\geq rn$.

We next prove that the same coefficient occurs in the cohomology ring.  We
use rational Cherednik algebras only to lift the relation
\eqref{eq:boundary-fixed} to an equivariant relation.  In particular, the
graded-algebra isomorphism proved in \cite{krylov-shlykov} is not used in the
argument below.

\begin{theorem}\label{thm:boundary-cohomology}
Let $a,a'\geq0$ and $0\leq c,c'<r$, and suppose that
$(a,c),(a',c')\neq(0,0)$.  If
\[
 a+a'=n-1\qquad\text{and}\qquad c+c'<r,
\]
then
\begin{equation}
 q_{ra+c}q_{ra'+c'}=nq_{r(n-1)+c+c'}
 \quad\text{in }H^*(M_{r,n},\mathbb C).              \label{eq:boundary-cohomology}
\end{equation}
\end{theorem}

\begin{proof}
We give the lifting argument, including the external results that it uses.
Let
\[
 \mathbf k=
 \mathbb C[\kappa,\alpha_1,\ldots,\alpha_r]/
 (\alpha_1+\cdots+\alpha_r),
 \qquad
 \mathfrak m=(\kappa,\alpha_1,\ldots,\alpha_{r-1}).
\]
Let $H_{n,r}$ be the rational Cherednik algebra for
$S_n\ltimes(\mathbb Z/r\mathbb Z)^n$ at Planck parameter zero.  Let
$\mathbf e=|G_n|^{-1}\sum_{g\in G_n}g$ and
$R=\mathbf eH_{n,r}\mathbf e$.  We use the conventions of
\cite[Definition~3.6]{krylov-shlykov}.  More precisely, for
$\eta=\exp(2\pi\sqrt{-1}/r)$ the color parameters are identified by
\[
 \alpha_\beta=P(\eta^{\beta-1}),\qquad
 P(z)=\frac1r\sum_{h=1}^{r-1}\frac{c_h}{\eta^{-h}-1}z^h.
\]
For $r=1$ the sum is empty.  This is an invertible linear change from
$c_1,\ldots,c_{r-1}$ to the parameters $\alpha_\beta$ subject to
$\sum_\beta\alpha_\beta=0$.
By \cite[Theorems~1.3, 1.6 and~3.1]{etingof-ginzburg}, including the
parameter version of the Satake isomorphism recalled in
\cite[Proposition~3.9]{krylov-shlykov}, $R$ is commutative, is identified
with $Z(H_{n,r})$, and its fiber at
$\mathfrak m$ is $\mathbb C[\mathbb C^{2n}]^{G_n}$.

Let $R=\bigoplus_{w\in\mathbb Z}R_w$ be the grading corresponding to the
$\mathbb T$-action and put
\[
 Q=R/(R_w\mid w\neq0).
\]
The parameters have $\mathbb T$-weight zero.  Thus every nonzero-weight
element of $R/\mathfrak mR$ lifts to one of the same weight in $R$.
It follows, by taking the quotient by the two ideals in either order, that
\begin{equation}
 Q/\mathfrak mQ\simeq
 \mathbb C[Z_n^{\mathbb T}].                         \label{eq:special-fiber-Q}
\end{equation}

Let $u_1,\ldots,u_n$ be the Dunkl--Opdam elements of
\cite[Definition~6.3]{krylov-shlykov}.  These elements commute, and we denote
by $U_j$ the image, under Satake and then in $Q$, of the central element
\[
 \sum_{i=1}^n u_i^j.
\]
The symmetric polynomials in the $u_i$ are central by
\cite[Lemma~6.5]{krylov-shlykov}.  Moreover,
\cite[Definition~6.3]{krylov-shlykov} gives
\begin{equation}
 U_j\bmod\mathfrak m=r^{-j}p_j.                      \label{eq:leading-U}
\end{equation}
The algebra in \eqref{eq:special-fiber-Q} is generated by
$p_1,\ldots,p_n$: indeed, it is generated by the power sums $p_j$, and the
power sums of degree greater than $n$ are polynomials in the first $n$ power
sums by the Newton identities.  Give the Cherednik generators $x_i,y_i$
degree one, the group elements degree zero, and
the parameters $\kappa,\alpha_i$ degree two.  Graded induction applied to
\eqref{eq:special-fiber-Q} and \eqref{eq:leading-U} then gives
\begin{equation}
 Q=\mathbf k[U_1,\ldots,U_n].                         \label{eq:Q-generated}
\end{equation}
More explicitly, after subtracting a polynomial in the $U_j$ with the same
image modulo $\mathfrak m$, a homogeneous element of degree $d$ belongs to
$\mathfrak mQ$.  It is a sum of the parameters multiplied by elements
of degree $d-2$, to which induction applies.  The induction starts from
$Q_0=\mathbb C$.  In particular, no division by a parameter, or flatness
assumption on $Q$, is involved.

Set $k=ra+c$, $l=ra'+c'$ and $D=k+l$.  By
Proposition~\ref{prop:pair-product}, \eqref{eq:special-fiber-Q}, and
\eqref{eq:leading-U},
\[
 U_kU_l-nU_D\in\mathfrak mQ.
\]
Using \eqref{eq:Q-generated}, we can therefore write
\begin{equation}
 U_kU_l-nU_D
 =\kappa F_0(U_1,\ldots,U_n)
  +\sum_{i=1}^{r-1}\alpha_iF_i(U_1,\ldots,U_n)       \label{eq:lifted-relation}
\end{equation}
for some polynomials $F_i$ with coefficients in $\mathbf k$.

It remains to transfer \eqref{eq:lifted-relation} to equivariant
cohomology.  Let $T_r$ be the diagonal maximal torus of $\operatorname{SL}_r$
acting on the framing, and let $A=\mathbb T\times T_r$ be the torus of
\cite[Section~4]{krylov-shlykov}; the first factor acts on the two plane
coordinates with opposite weights.  Thus $H_A^*(\operatorname{pt})=\mathbf k$.
The $A$-fixed points of
$M_{r,n}$ are indexed by $r$-multipartitions
$\boldsymbol\lambda=(\lambda^1,\ldots,\lambda^r)$ of $n$.  At such a fixed
point, put
\[
 s_j(\boldsymbol\lambda)
 =\sum_{\beta=1}^r\sum_{\square\in\lambda^\beta}
 (\alpha_\beta+\kappa\operatorname{ct}(\square))^j.
\]
Here $\operatorname{ct}(i,j)=j-i$.
As in \cite[Section~6.1]{krylov-shlykov}, the highest $\mathbb T$-graded
component of the standard module $\Delta(\boldsymbol\lambda)$ is
$S(\boldsymbol\lambda)\otimes\mathbf k$, where $S(\boldsymbol\lambda)$ is
the corresponding irreducible $G_n$-module.  Weight-zero central elements
act there by scalars in $\mathbf k$.  Positive-weight central elements
annihilate this component, so every product of central elements of
opposite nonzero weights also annihilates it.  Hence this scalar action
factors through a $\mathbf k$-algebra homomorphism
$\chi_{\boldsymbol\lambda}:Q\to\mathbf k$.
By \cite[Lemma~6.7]{krylov-shlykov}, it satisfies
$\chi_{\boldsymbol\lambda}(U_j)=s_j(\boldsymbol\lambda)$.
Applying this character to \eqref{eq:lifted-relation} gives
\begin{align}
 s_k(\boldsymbol\lambda)s_l(\boldsymbol\lambda)
 -ns_D(\boldsymbol\lambda)
={}&\kappa F_0(s_1(\boldsymbol\lambda),\ldots,
 s_n(\boldsymbol\lambda))\notag\\
 &+\sum_{i=1}^{r-1}\alpha_i
 F_i(s_1(\boldsymbol\lambda),\ldots,
 s_n(\boldsymbol\lambda)).                          \label{eq:fixed-relation}
\end{align}

On the other hand, the restriction of the equivariant class
$q_j^A=j!\operatorname{ch}_j^A(\mathcal V)$ to the same fixed point is
$s_j(\boldsymbol\lambda)$ by
\cite[Lemma~4.7]{krylov-shlykov}.  The fixed-point restriction
\[
 H_A^*(M_{r,n},\mathbb C)\longrightarrow
 \bigoplus_{\boldsymbol\lambda}\mathbf k
\]
is injective by \cite[Lemma~4.5]{krylov-shlykov}.  Hence
\eqref{eq:fixed-relation} implies in equivariant cohomology that
\begin{align*}
 q_k^Aq_l^A-nq_D^A
={}&\kappa F_0(q_1^A,\ldots,q_n^A)
 +\sum_{i=1}^{r-1}\alpha_iF_i(q_1^A,\ldots,q_n^A).
\end{align*}
Finally, the forgetful homomorphism to ordinary cohomology sends
$q_j^A$ to $q_j$ and all the parameters to zero.  Applying it gives
\eqref{eq:boundary-cohomology}.
\end{proof}

\begin{rem}\label{rem:scope-boundary}
The proof of Theorem~\ref{thm:boundary-cohomology} uses the PBW and Satake
theorems, the centrality and eigenvalue formula for the Dunkl--Opdam
elements, and equivariant localization.  It does not use
\cite[Theorem~1.22]{krylov-shlykov}, which identifies the two algebras,
or the injectivity of the character map in
\cite[Proposition~6.8]{krylov-shlykov}.
Nevertheless, the use of the same Cherednik-algebra deformation means that
this argument should not be regarded as a method independent of
\cite{krylov-shlykov}.  Its purpose here is the explicit comparison of the
coefficient $n$ in the product of two generators.
\end{rem}

Kohei Hatano

Department of Mathematics, Graduate School of Science

Hokkaido University

Sapporo, Hokkaido, 0600810

Japan

E-mail: hatano@math.sci.hokudai.ac.jp
\end{document}